\documentclass[10pt]{article}

\usepackage{hyperref}
\usepackage{xcolor}
\usepackage{amssymb}
\usepackage{amsfonts}
\usepackage{amsthm}
\usepackage{amsmath}
\usepackage{float}
\usepackage{bigints}
\usepackage{fullpage}
\usepackage{mathtools}
\usepackage[utf8]{inputenc} 
\usepackage{cancel}
\usepackage{verbatim}

\newtheorem{Lemma}{Lemma}

\newtheorem{Theorem}{Theorem}
\newtheorem{Conjecture}{Conjecture}
\newtheorem{Definition}{Definition}
\newtheorem{Problem}{Problem}
\newtheorem{Corollary}{Corollary}
\newtheorem{Remark}{Remark}

\numberwithin{Subcase}{Case}

\DeclareMathOperator{\proj}{proj}
\newcommand{\E}{\mathbb{E}}

\renewcommand{\Re}{\mathbb R}

\newcommand{\F}{\mathcal{F}}
\renewcommand{\S}{\mathbb{S}}

\renewcommand{\c}{\mathbf{c}}
\renewcommand{\a}{\mathbf{a}}
\renewcommand{\b}{\mathbf{b}}
\newcommand{\m}{\mathbf{m}}
\newcommand{\q}{\mathbf{q}}
\newcommand{\p}{\mathbf{p}}

\renewcommand{\o}{\mathbf{o}}
\newcommand{\x}{\mathbf{x}}
\newcommand{\y}{\mathbf{y}}

\DeclareMathOperator{\conv}{conv}

\DeclareMathOperator{\area}{area}

\DeclareMathOperator{\Sarea}{Sarea}

\title{From the separable Tammes problem to extremal distributions of great circles in the unit sphere
\footnote{\hspace*{.35cm} 2010 Mathematics Subject Classification: 52A55, 52A40, 52C15.}}
\author{K\'{a}roly Bezdek\thanks{Partially supported by a Natural Sciences and 
Engineering Research Council of Canada Discovery Grant.} and Zsolt L\'angi\thanks{Partially supported by the National Research, Development and Innovation Office, NKFI, K-119670 and BME Water Sciences \& Disaster Prevention TKP2020 Institution Excellence Subprogram, grant no. TKP2020 BME-IKA-VIZ.}
}

\date{}

\begin{document}

\maketitle

\begin{abstract}

A family of spherical caps of the 2-dimensional unit sphere $\mathbb{S}^2$ is called a totally separable packing in short, a TS-packing if any two spherical caps can be separated by a great circle which is disjoint from the interior of each spherical cap in the packing. The separable Tammes problem asks for the largest density of given number of congruent spherical caps forming a TS-packing in $\mathbb{S}^2$. We solve this problem up to $8$ spherical caps and upper bound the density of any TS-packing of congruent spherical caps in terms of their angular radius. Based on this, we show that the centered separable kissing number of $3$-dimensional Euclidean balls is $8$. Furthermore, we prove bounds for the maximum of the smallest inradius of the cells of the tilings generated by $n>1$ great circles in $\mathbb{S}^2$. Next, we prove dual bounds for TS-coverings of $\mathbb{S}^2$ by congruent spherical caps. Here a covering of $\mathbb{S}^2$ by spherical caps is called a totally separable covering in short, a TS-covering if there exists a tiling generated by finitely many great circles of $\mathbb{S}^2$ such that the cells of the tiling are covered by pairwise distinct spherical caps of the covering. Finally, we extend some of our bounds on TS-coverings to spherical spaces of dimension $>2$. 

\end{abstract}

\section{Introduction}\label{sec:intro}

The Tammes problem, which is one of the best-known problems of discrete geometry, was originally proposed by the Dutch botanist Tammes  \cite{Ta} in 1930. It is about finding the arrangement of $N$ points on a unit sphere such that it maximizes the minimum distance between any two points. Equivalently, for given $N>1$ one is asking for the largest $r$ such that there exists a packing of $N$ spherical caps with angular radius $r$ in the unit sphere. This problem is solved for several values of $N$ namely, for $N = 3,4,6,12$ by L. Fejes T\'oth \cite{LFT}; for $N = 5,7,8,9$ by Sch\"utte and van der Waerden \cite{SchWa}; for $N = 10,11$ by Danzer \cite{Da} (see also \cite{Bo}, \cite{Ha}); for $N=13, 14$ by Musin and Tarasov \cite{MuTa12}, \cite{MuTa15}; and for $N = 24$ by Robinson \cite{Ro}. For a comprehensive overview on the Tammes problem we refer the interested reader to the recent articles \cite{FGT} and \cite{Mu}.

In this paper we start our investigations with the Tammes problem for totally separable packings (Problem~\ref{separable-Tammes-1}) and then we turn our attention to the closely related problems on extremal distributions of great circles and totally separable coverings (Problems~\ref{separable-Tammes-2} and~\ref{separable-coverings}). Finally, we conclude our investigations with some higher dimensional analogue statements. We note that totally separable packings and coverings have been intensively studied in Euclidean (resp., normed) spaces (\cite{Bez}, \cite{BeSzSz}, \cite{BeNa}, \cite{BeKhOl}, \cite{BezLan}, \cite{BeLa19}, \cite{BeLa20}, \cite{Be21}, \cite{FeFe}, \cite{FeFe87}, \cite{Ke}), but not yet in spherical spaces, which is the main target of this paper. The details are as follows. 

Let $\S^2=\{\mathbf{x}\in\E^{3} \ |\ \|\mathbf{x}\|=\sqrt{\langle\mathbf{x},\mathbf{x}\rangle}=1\}$ be the unit sphere centered at the origin $\o$ in the $3$-dimensional Euclidean space $\E^{3}$, where $\|\cdot\|$ and $\langle \cdot , \cdot \rangle$ denote the canonical Euclidean norm and the canonical inner product in $\E^{3}$. A {\it great circle} of $\S^2$ is an intersection of $\S^2$ with a plane of $\E^{3}$ passing through $\o$.  Two points are called {\it antipodes} if they can be obtained as an intersection of $\S^2$ with a line through $\o$ in $\E^{3}$. If $\a ,\b\in\S^2$ are two distinct points that are not antipodes, then we label the (uniquely determined) shortest geodesic arc of $\S^2$ connecting $\a$ and $\b$ by $\widehat{\a\b}$. In other words, $\widehat{\a\b}$ is the shorter circular arc with endpoints $\a$ and $\b$ of the great circle $\a\b$ that passes through $\a$ and $\b$. The length of $\widehat{\a\b}$ is called the spherical distance between $\a$ and $\b$ and it is labelled by $l(\widehat{\a\b})$, where $0<l(\widehat{\a\b})<\pi$. If $\a,\b\in\S^2$ are antipodes, then we set $l(\widehat{\a\b})=\pi$. Let $\x\in\S^2$ and $r\in (0,\frac{\pi}{2}]$. Then the set $C[\x, r]:=\{\y\in\S^2\ |\ l(\widehat{\x\y})\leq r\}=\{\y\in\S^2\ |\ \langle\y ,\x\rangle\geq\cos r\}$
is called the (closed) {\it spherical cap}, centered at $\x$ having angular radius $r$ in $\S^2$. $C[\x, \frac{\pi}{2}]$ is called a (closed) {\it hemisphere}. A {\it packing} of spherical caps in $\S^2$ is a family of spherical caps having pairwise disjoint interiors. 

\subsection{On the densest TS-packings by congruent spherical caps in $\S^2$}

The following definition is a natural extension to $\S^2$ of the Euclidean analogue notion, which was introduced by G. Fejes T\'oth and L. Fejes T\'oth \cite{FeFe} and has attracted significant attention.

\begin{Definition}
A family of spherical caps of $\S^2$ is called a {\rm totally separable packing} in short, a {\rm TS-packing} if any two spherical caps can be separated by a great circle of $\S^2$ which is disjoint from the interior of each spherical cap in the packing. 
\end{Definition}

Now, we raise the immediate analogue of the Tammes problem for TS-packings as follows.

\begin{Problem}\label{separable-Tammes-1}
{\bf (Separable Tammes Problem)}
For given $k>1$ find the largest $r>0$ such that there exists a TS-packing of $k$ spherical caps with angular radius $r$ in $\S^2$. Let us denote this $r$ by $r_{\rm STam}(k, \S^2)$.
\end{Problem}

\begin{Remark}\label{central-symmetry}
Let $\mathcal{P}$ be a TS-packing of an odd number of spherical caps of angular radius $r$ in  $\S^2$. Then there exists a family $\mathcal {C}$ of great circles of minimal cardinality such that any two caps of $\mathcal{P}$ can be separated by a great circle of $\mathcal {C}$ without intersecting the interior of any spherical cap of $\mathcal{P}$. Hence, the great circles of $\mathcal {C}$ dissect $\S^2$ into an even number of $2$-dimensional cells forming an $\o$-symmetric tiling such that each cap of $\mathcal{P}$ is contained in exactly one cell and no two caps of $\mathcal{P}$ belong to the same cell. As $\mathcal{P}$ is a packing of an odd number of caps therefore, one can always add an additonal cap of angular radius $r$ to $\mathcal{P}$ and thereby obtain a TS-packing of even number of caps. Thus, it follows that for any integer $k'>1$, we have $r_{\rm STam}(2k'-1, \S^2)= r_{\rm STam}(2k', \S^2)$.
\end{Remark}

Consider three mutually orthogonal great circles on $\S^2$. These divide the sphere into eight regular spherical triangles of side lengths $\frac{\pi}{2}$. The family of inscribed spherical caps of these triangles is a TS-packing of eight spherical caps of radius $\arcsin \frac{1}{\sqrt{3}} (\approx 35.26^{\circ})$. We call such a family an \emph{octahedral TS-packing} (Figure~\ref{octahedral}). Similarly, the side lines of a regular spherical triangle of side length $\arccos \frac{1}{4} \approx 75.52^{\circ}$ divide the sphere into two regular spherical triangles of side length $\arccos \frac{1}{4}$, and six isosceles spherical triangles of side lengths $\pi - \arccos \frac{1}{4}, \pi - \arccos \frac{1}{4}$ and $\arccos \frac{1}{4}$. 
The inscribed spherical caps of the six isosceles triangles form a TS-packing of six spherical caps of radius $\arctan \frac{3}{4} \approx 36.87^{\circ}$, which we call a \emph{cuboctahedral TS-packing} (Figure~\ref{cuboctahedral}).

\begin{figure}[h]
\begin{minipage}[t]{0.48\linewidth}
\centering
\includegraphics[width=0.7\textwidth]{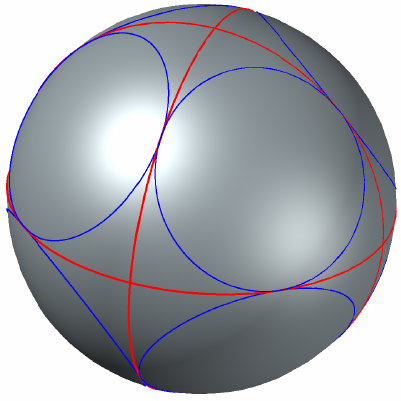}
\caption{An octahedral TS-packing in $\S^2$.}
\label{octahedral}
\end{minipage}
\begin{minipage}[t]{0.48\linewidth}
\centering
\includegraphics[width=0.7\textwidth]{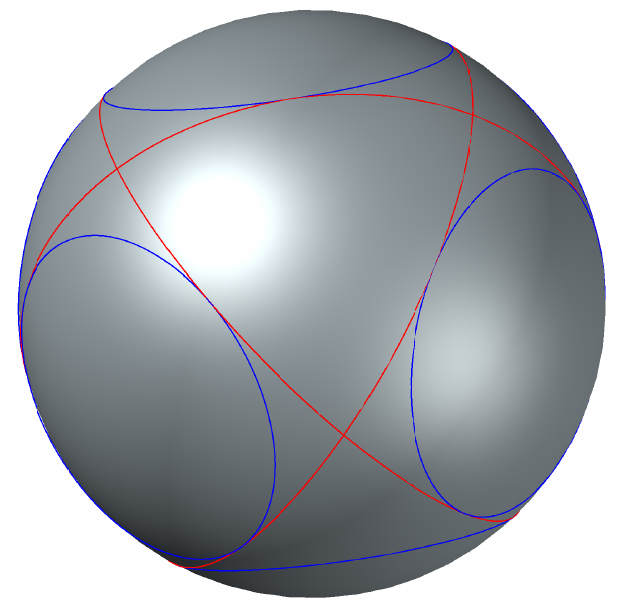}
\caption{A cuboctahedral TS-packing in $\S^2$.}
\label{cuboctahedral}
\end{minipage}
\end{figure}

We leave the easy proofs of $r_{\rm STam}(2, \S^2)=\frac{\pi}{2}(=90^{\circ})$ and $r_{\rm STam}(3, \S^2)=r_{\rm STam}(4, \S^2)=\frac{\pi}{4}(=45^{\circ})$ to the reader. Here we solve Problem~\ref{separable-Tammes-1} for $k=5,6,7,8$ moreover, bound $r_{\rm STam}(k, \S^2)$ for $k>8$ as follows.

\begin{Definition}
Let $k >1$ be fixed. A TS-packing of $k$ spherical caps of radius $r_{\rm STam}(k, \S^2)$ is called \emph{$k$-optimal}.
\end{Definition}

\begin{Theorem}\label{thm:few_caps}
For $5 \leq k \leq 6$ we have $r_{\rm STam}(k, \S^2)= \arctan \frac{3}{4}$, and any $k$-optimal TS-packing is a subfamily of a cuboctahedral TS-packing. Furthermore, for $7 \leq k \leq 8$ we have $r_{\rm STam}(k, \S^2)= \arcsin \frac{1}{\sqrt{3}}$, and any $k$-optimal TS-packing is a subfamily of an octahedral TS-packing.\footnote{After completing this manuscript, G. Fejes T\'oth has informed us that in 1981 in a seminar talk of K. B\"or\"oczky the problem of finding  $r_{\rm STam}(k, \S^2)$ for $k\leq 12$ was raised and investigated. Unfortunately, that talk has not been published. Furthermore, we thank G. Fejes T\'oth for sending us the article \cite{Va} of \'E. V\'as\'arhelyi that proves Part (ii) of Theorem~\ref{main4}.}
\end{Theorem}

\begin{Definition}\label{defn:density}
Let $\F_m = \{ S_1, S_2, \ldots, S_m \}$ be a TS-packing of spherical caps of radius $\rho < \frac{\pi}{4}$ on $\S^2$. The \emph{density} $\delta(\F_m)$ of $\F_m$ is defined as
\[
\delta(\F_m) := \frac{\Sarea(\bigcup_{i=1}^m S_i)}{\Sarea(\S^2)} = \frac{2m\pi(1-\cos \rho)}{4 \pi} = \frac{1-\cos \rho}{2} m,
\]
where $\Sarea(\cdot)$ refers to the spherical area (i.e., spherical Lebesgue measure) in $\S^2$.
\end{Definition}

\begin{Definition}\label{defn:density_TSpacking}
Let $L_1$ and $L_2$ be two great circles on $\S^2$ orthogonally intersecting at the points $\pm\p$. Let $S_1$, $S_2$, $S_3$ be three spherical caps of radius $\rho < \frac{\pi}{4}$, each touching both $L_1$ and $L_2$ such that the touching points are closer to $\p$ than $-\p$. Let $T$ denote the spherical convex hull of the centers of $S_1$, $S_2$ and $S_3$. Then we set
\[
\delta(\rho) := \frac{\Sarea(\bigcup_{i=1}^3 (S_i \cap T))}{\Sarea(T)}=\frac{1-\cos\rho}{1-\frac{\pi}{4\arcsin\left(\frac{1}{\sqrt{2}\cos\rho}\right)}}.
\]
\end{Definition}

\begin{Theorem}\label{main2}\text{}
\begin{itemize}
\item[(i)] Let $\F_m = \{ S_1, S_2, \ldots, S_m \}$ be a TS-packing of spherical caps of radius $\rho < \frac{\pi}{4}$ on $\S^2$. Then $\delta(\F_m) \leq \delta(\rho)$ and therefore $r_{\rm STam}(k, \S^2)\leq\arccos\frac{1}{\sqrt{2}\sin\left(\frac{k}{k-2}\frac{\pi}{4}\right)}$ for all $k\geq 5$. In particular, $r_{\rm STam}(8, \S^2)=\arccos\sqrt{\frac{2}{3}}$ $=\arcsin\frac{1}{\sqrt{3}}(\approx 35.26^{\circ})$.
\item[(ii)] For any sufficiently large value of $m$, we have $r_{\rm STam}(m, \S^2) \geq \frac{0.793}{\sqrt{m}}$, or equivalently, there is a TS-packing $\F_m$ of $m$ congruent spherical caps with $\delta(\F_m) \geq m \frac{1-\cos \frac{0.793}{\sqrt{m}}}{2}$, where $\lim_{m\to+\infty}m \frac{1-\cos \frac{0.793}{\sqrt{m}}}{2}=\frac{0.793^2}{4}\approx 0.16<\frac{\pi}{4}\approx 0.79$.
\end{itemize}
\end{Theorem}

\begin{figure}[ht]
\begin{center}
\includegraphics[width=0.55\textwidth]{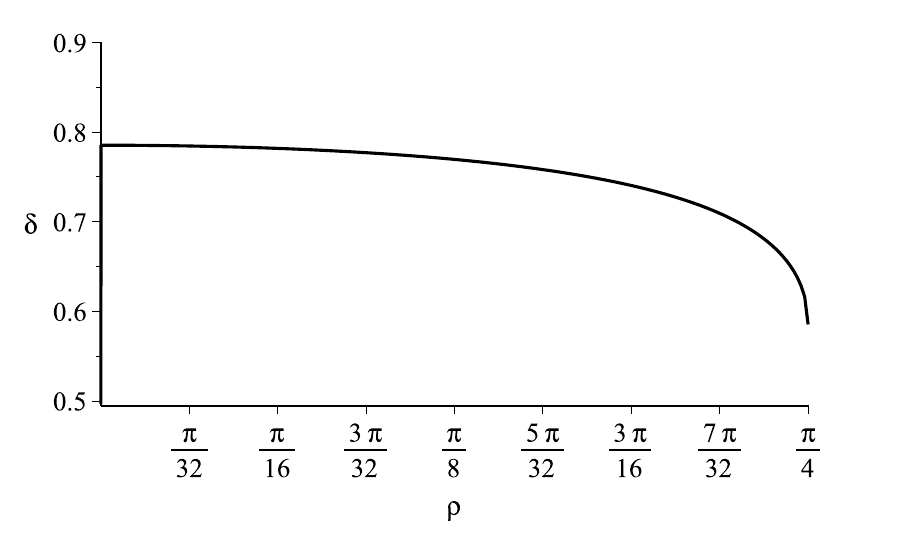}
\caption{The graph of the function $\delta(\rho)$ defined in Definition~\ref{defn:density_TSpacking}.}
\label{density}
\end{center}
\end{figure}

We note that an elementary computation, using L'Hospital's Rule, yields that $\delta(\rho)$ is strictly decreasing over $(0, \frac{\pi}{4})$ with $\lim_{\rho \to 0^{+}} \delta(\rho) = \frac{\pi}{4}$ (Figure~\ref{density}). Recall that a packing  $\mathcal {P}$ of disks is called a totally separable packing in short, a TS-packing in the Euclidean plane $\E^{2}$ if any two disks of $\mathcal {P}$ can be separated by a line of $\E^{2}$ such that it is disjoint from the interior of each disk in $\mathcal {P}$. As the proof of Theorem~\ref{main2} extends to the so-called $(2R_{\rho})$-separable packings of spherical caps of radius $\rho$ (see the relevant discussion at the beginning of the proof of Theorem~\ref{main2}), therefore a standard limiting process combined with central projection implies the following theorem of G. Fejes T\'oth and L. Fejes T\'oth \cite{FeFe}. It states that the largest density of totally separable packings of congruent disks in $\E^{2}$ is $\frac{\pi}{4}$. Thus, Theorem~\ref{main2} yields

\begin{Corollary}\label{constant density upper bound}
The density of any TS-packing of at least three congruent spherical caps in $\S^2$ is always strictly less than $\frac{\pi}{4}$, where $\frac{\pi}{4}$ is the largest density of TS-packings of congruent disks in $\E^{2}$. 
\end{Corollary}


\begin{Corollary}\label{upper bound for the separable kissing number}
As $r_{\rm STam}(12, \S^2))\leq\arccos\frac{1}{\sqrt{2}\sin\left(\frac{12}{10}\frac{\pi}{4}\right)}\approx 29.07^{\circ}$ and  $r_{\rm STam}(10, \S^2))\leq\arccos\frac{1}{\sqrt{2}\sin\left(\frac{10}{8}\frac{\pi}{4}\right)}\approx 31.74^{\circ}$, therefore the maximum number of spherical caps of angular radius $\frac{\pi}{6}$ that form a TS-packing in $\S^2$, is at most $10$.
\end{Corollary}

The problem of finding the largest number of spherical caps of angular radius $\frac{\pi}{6}$ that form a TS-packing in $\S^2$, can be rephrased as follows.

\begin{Definition}\label{separable kissing number}
Let $\tau_{\rm{csep}}(\E^{3})$ denote the maximum cardinality of a packing $\cal{P}$ of unit balls with each unit ball touching the unit ball centred at the origin $\o$ in $\E^{3}$ such that any two unit balls of $\cal{P}$ can be separated by a plane passing through $\o$ which is disjoint from the interiors of the unit balls of $\cal{P}$. We call $\tau_{\rm{csep}}(\E^{3})$ the {\rm centered separable kissing number} of unit balls in $\E^{3}$.
\end{Definition}

\begin{Remark}\label{equivalence}
The central projection of $\E^{3}$ from $\o$ onto $\S^2$ clearly implies that $\tau_{\rm{csep}}(\E^{3})$ is equal to the maximum number of spherical caps of angular radius $\frac{\pi}{6}$ that form a TS-packing in $\S^2$.
\end{Remark}

We close this section with

\begin{Theorem}\label{3D-sep-kiss-num}\label{sub-main2}
$\tau_{\rm{csep}}(\E^{3})=8$, which can be achieved by a TS-packing of $8$ spherical caps of angular radius $\frac{\pi}{6}$ possessing octahedral symmetry in $\S^2$.
\end{Theorem}

\subsection{Maximizing the smallest angular inradius of the cells of the tilings generated by $n$ great circles in $\S^2$.}

In this section, we raise and investigate Problem~\ref{separable-Tammes-2}, which is a variant of Problem~\ref{separable-Tammes-1}. Namely, instead of fixing the number of spherical caps, we fix the number of separating great circles. Before stating this problem and our results we recall that the {\it (angular) inradius} of a closed set lying on a closed hemisphere of $\S^2$ is the angular radius of the largest spherical cap contained in the closed set.

\begin{Problem}\label{separable-Tammes-2}
For given $n>1$ find the largest $r>0$ such that there exists an arrangement of $n$ pairwise distinct great circles in $\S^2$ with the property that each $2$-dimensional cell of the tiling of $\S^2$ generated by them contains a spherical cap of radius $r$. Let us denote this $r$ by $r_{\rm gc}(n,  \S^2)$. 
\end{Problem}

\noindent We solve Problem~\ref{separable-Tammes-2} for $n=2,3,4$ moreover, bound $r_{\rm gc}(n,  \S^2)$ for $n\geq 5$ as follows.

\begin{Definition}
Let $n >1$ be fixed. An arrangement of $n$ pairwise distinct great circles in $\S^2$ is called \emph{$n$-optimal} if each cell of the tiling of $\S^2$ generated by them contains a spherical cap of radius $r_{\rm gc}(n,  \S^2)$.
\end{Definition}

\begin{Theorem}\label{main3}
Let $\mathcal{F}_n = \{ G_1, G_2 \ldots, G_n \}$ be a family of $n$ great circles. Then,
\begin{itemize}
\item[(i)] $r_{\rm gc}(2,  \S^2)= \frac{\pi}{4}$, and $\mathcal{F}_2$ is $2$-optimal if and only if $G_1$ and $G_2$ are orthogonal;
\item[(ii)] $r_{\rm gc}(3,  \S^2) = \arcsin \frac{1}{\sqrt{3}} \approx 35.26^{\circ}$, and $\mathcal{F}_3$ is $3$-optimal if and only if $G_1$, $G_2$ and $G_3$ are mutually orthogonal;
\item[(iii)] $r_{\rm gc}(4,  \S^2)= \arcsin \frac{1}{\sqrt{5}} \approx 26.57^{\circ}$, and $\mathcal{F}_4$ is $4$-optimal if and only if three of the great circles, say $G_1, G_2$ and $G_3$ meet at the same pair of antipodes $\pm \mathbf{p}$ at angles equal to $\frac{\pi}{3}$, and $G_4$ is the polar of $\pm \mathbf{p}$ (cf. Figure~\ref{fig:thm3});
\item[(iv)] $r_{\rm gc}(n,  \S^2)\leq \arccos\left(\frac{1}{\sqrt{2}\sin\left(\frac{n}{n-1}\frac{\pi}{4}\right)}\right)$ holds for $n> 4$.
\end{itemize}
\end{Theorem}

\begin{figure}[ht]
\begin{center}
 \includegraphics[width=0.37\textwidth]{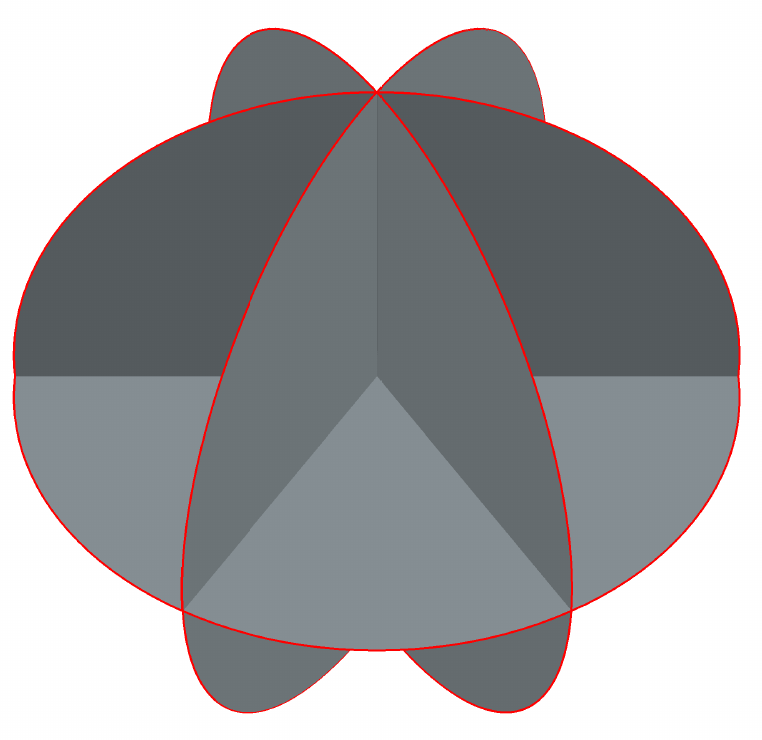}
 \caption{A $4$-optimal arrangement of great circles in $\S^2$.}
\label{fig:thm3}
\end{center}
\end{figure}

\begin{Remark}
The $n$ reflection planes of a regular (n-1)-sided right prism inscribed $\S^2$ generate $n$ great circles on $\S^2$ such that each cell of the tiling induced by those $n$ great circles contains a spherical cap of radius 
\begin{equation}\label{prism}
\arctan \left(  \sin\left(\frac{1}{n-1}\frac{\pi}{2}\right)  \right)  \leq r_{\rm gc}(n,  \S^2),
\end{equation}
where $n\geq 5$. However, for $n=6, 9, 15$ one can improve the lower bound of (\ref{prism}) by taking the reflection planes of a regular tetrahedron, octahedron, and icosahedron inscribed $\S^2$. Indeed, in these cases the relevant great circles dissect $\S^2$ into congruent spherical triangles obtained as the barycentric subdivision of the tetrahedral (Figure~\ref{tetrahedral-Coxeter}), octahedral (Figure~\ref{octahedral-Coxeter}), and icosahedral (Figure~\ref{icosahedral-Coxeter}) spherical mosaic. Thus, the angles of such a triangle $T_n$ are $\frac{\pi}{2}$, $\frac{\pi}{3}$ and $\frac{\pi}{k_n}$, where $k_6=3$, $k_9=4$ and $k_{15}=5$. For $n=6,9,15$, let $\rho_n$ denote the inradius of $T_n$. Then, an elementary computation yields that
\[
(18.44^{\circ} \approx) \arctan\frac{1}{3} = \rho_6 > \arctan \left(  \sin\left(\frac{1}{5}\frac{\pi}{2}\right) \right) \approx 17.17^{\circ},
\]
\[
(12.46^{\circ} \approx) \arctan \sqrt{\frac{2-\sqrt{2}}{12}} = \rho_9 > \arctan \left(  \sin\left(\frac{1}{8}\frac{\pi}{2}\right)  \right) \approx 11.04^{\circ}, and
\]
\[
(7.56^{\circ} \approx) \arccos \sqrt{\frac{210+12\sqrt{5}}{241}} = \rho_{15} > \arctan \left(  \sin\left(\frac{1}{14}\frac{\pi}{2}\right)  \right) \approx 6.39^{\circ}.
\]
\end{Remark}
So, it is natural to raise 
\begin{Conjecture}\label{Coxeter-type}
$r_{\rm gc}(6,  \S^2)=\rho_6, r_{\rm gc}(9,  \S^2)= \rho_9$, and $r_{\rm gc}(15,  \S^2)=\rho_{15}$.
\end{Conjecture}

\begin{figure}[h]
\begin{minipage}[t]{0.32\linewidth}
\centering
\includegraphics[width=0.9\textwidth]{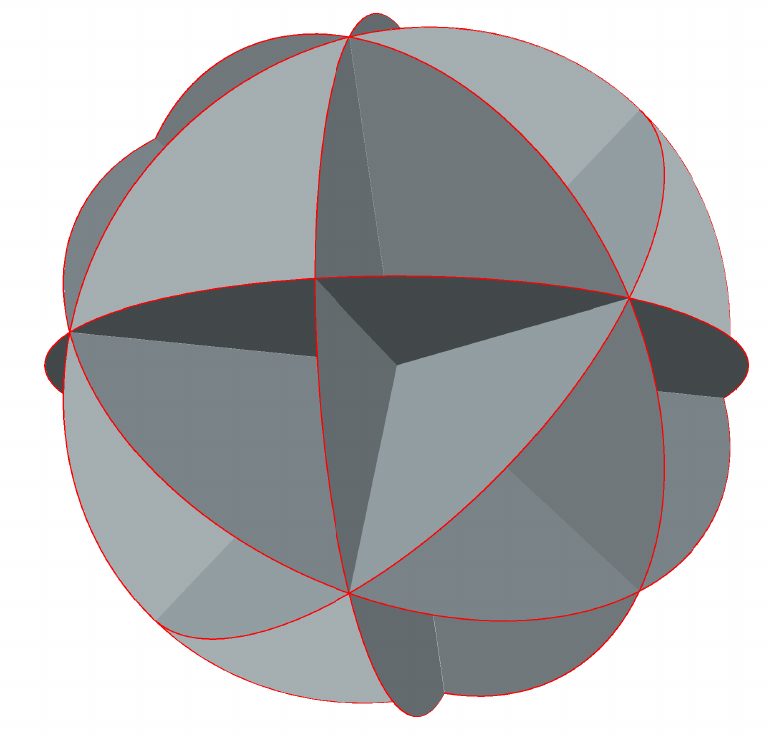}
\caption{$r_{\rm gc}(6,  \S^2)\geq\rho_6$}
\label{tetrahedral-Coxeter}
\end{minipage}
\begin{minipage}[t]{0.32\linewidth}
\centering
\includegraphics[width=0.9\textwidth]{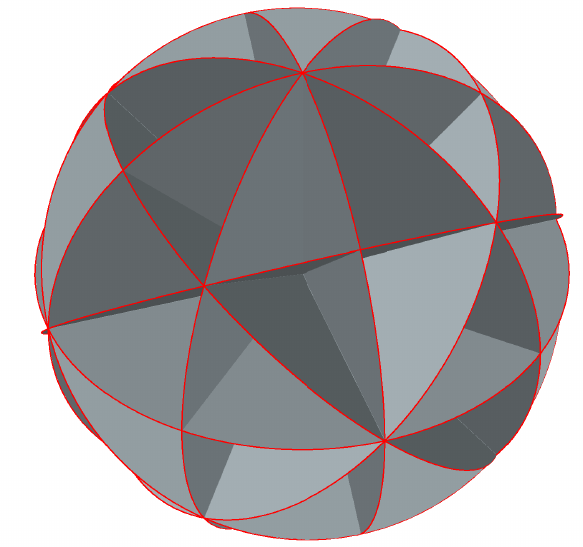}
\caption{$r_{\rm gc}(9,  \S^2)\geq\rho_9$}
\label{octahedral-Coxeter}
\end{minipage}
\begin{minipage}[t]{0.32\linewidth}
\centering
\includegraphics[width=0.85\textwidth]{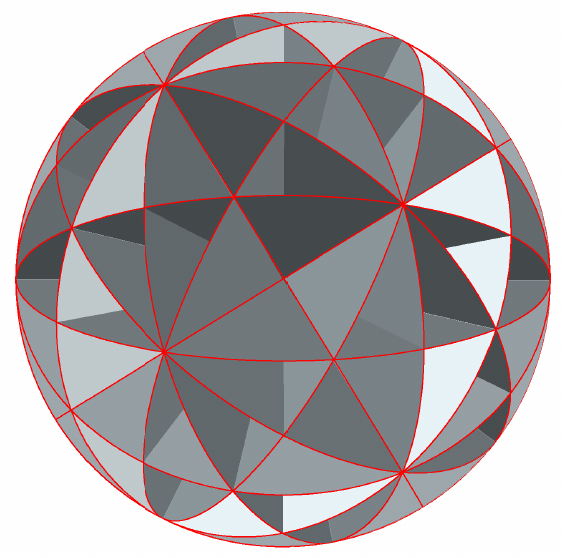}
\caption{$r_{\rm gc}(15,  \S^2)\geq\rho_{15}$}
\label{icosahedral-Coxeter}
\end{minipage}
\end{figure}

\begin{Remark}
Problem~\ref{separable-Tammes-2} suggests to raise the following close relative: for given $n>1$ {\rm minimize} the {\rm largest} inradius of the cells of tilings generated by $n$ great circles in $\S^2$. However, this problem is not new. It has been raised by L. Fejes T\'oth in \cite{FTL73} also in the following equivalent form: Let $n$ be a positive integer. Find the smallest spherical distance $\omega$ such that there exist $n$ zones of width $\omega$ covering $\S^2$. Here a zone of width $\omega$ on $\S^2$ is the set of points within spherical distance $\frac{\omega}{2}$ of a given great circle. Very recently Jiang and Polyanskii \cite{JP} have proved L. Fejes T\'oth's conjecture stating that if $n$ zones of width $\omega$ cover $\S^2$, then $\omega\geq \frac{\pi}{n}$. Ortega-Moreno \cite{OM} has found another proof, which was simplified by Zhao \cite{Zh} (see also \cite{GlKa}).
\end{Remark}

\subsection{Minimizing the largest circumradius of the cells of the tilings generated by $n$ great circles in $\S^2$}

Next, we raise and investigate Problem~\ref{separable-coverings}, which looks for the covering of $\S^2$ by congruent spherical caps of minimal angular radius under the condition that each $2$-dimensional cell of the underlying tiling generated by $n$ great circles is covered by some spherical cap of the covering. Before putting forward this question and stating our results we recall that the {\it circumradius} of a closed set lying on a closed hemisphere of $\S^2$ is the angular radius of the smallest spherical cap containing the closed set.

\begin{Problem}\label{separable-coverings}
For given $n>1$ find the smallest $R>0$ such that there exists an arrangement of $n$ great circles in $\S^2$ with the property that each $2$-dimensional cell of the tiling of $\S^2$ generated by them can be covered by a spherical cap of radius $R$. Let us denote this $R$ by $R_{\rm gc}(n, \S^2)$. 
\end{Problem}

\begin{Definition}
Let $n >1$ be fixed. An arrangement of $n$ great circles in $\S^2$ is called \emph{$n$-extremal} if each $2$-dimensional cell of the tiling of $\S^2$ generated by them can be covered by a spherical cap of radius $R_{\rm gc}(n, \S^2)$.
\end{Definition}

Clearly, $R_{\rm gc}(2, \S^2)=\frac{\pi}{2}(=90^{\circ})$. On the other hand, we have

\begin{Theorem}\label{main4} \text{}
\begin{itemize}
\item[(i)] $R_{\rm gc}(3, \S^2)= \arccos \frac{1}{\sqrt{3}}$, and a family of $3$ great circles is $3$-extremal in $\S^2$ if and only if it consists of mutually orthogonal great circles.

\item[(ii)] $R_{\rm gc}(4, \S^2)=\frac{\pi}{4}$, and a family of $4$ great circles is $4$-extremal in $\S^2$ if and only if the double normals of the great circles form the vertices of a cube inscribed $\S^2$ (Figure~\ref{fig:cuboctahedral}). 

\item[(iii)] $\arccos\left(\frac{1}{\tan\left(\left(1+\frac{2}{n^2-n+2}\right)\frac{\pi}{4}\right)  }\right)\leq R_{\rm gc}(n, \S^2) $ holds for $n> 3$.
\item[(iv)] $R_{\rm gc}(n, \S^2)$ $\leq \arcsin \left( \frac{8}{\sqrt{3} \sqrt[3]{n}} \right)$ holds for $n> 3$.
\end{itemize}
\end{Theorem}

\begin{figure}[ht]
\begin{center}
\includegraphics[width=0.37\textwidth]{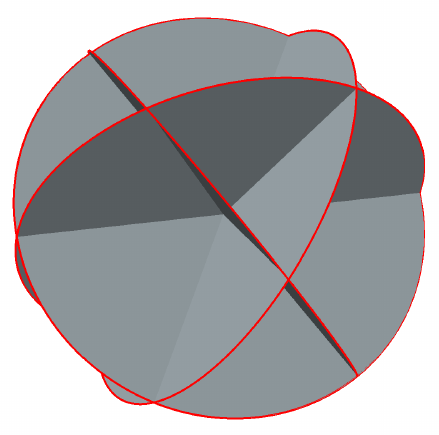}
\caption{ The largest circumradius of the $2$-dimensional cells of the cuboctahedral tiling generated by $4$ extremal great circles in Part (ii) of Theorem~\ref{main4} is equal to $\frac{\pi}{4}$. }
\label{fig:cuboctahedral}
\end{center}
\end{figure}

\subsection{On the thinnest TS-coverings by congruent spherical caps in $\S^2$}

G. Fejes T\'oth  \cite{FeFe87} introduced the notion of totally separable coverings in the Euclidean plane. We extend this concept to $\S^2$ as follows.

\begin{Definition}\label{TS-covering}
We say that a covering $\mathcal{C}$ of $\S^2$ by (congruent) spherical caps is a {\rm totally separable covering} in short, a {\rm TS-covering} if there exists a tiling $\mathcal{T}$ generated by finitely many great circles of $\S^2$ such that the $2$-dimensional cells of $\mathcal{T}$ are covered by pairwise distinct spherical caps of $\mathcal{C}$.\footnote{According to  \cite{FeFe87}, one would need also the condition that every spherical cap of $\mathcal{C}$ covers some $2$-dimensional cell of $\mathcal{T}$. However, we are interested only in lower bounding the density of TS-coverings and so, we do not impose this extra condition.}
\end{Definition}

Furthermore, we need 

\begin{Definition}\label{covering-density}
If $\mathcal{C}$ is a covering of $\S^2$ by the spherical caps $C_1,C_2,\dots , C_m$ (of angular radius $\leq \frac{\pi}{2}$), then the {\rm density} $\delta(\mathcal{C})$ of $\mathcal{C}$ is defined by
$$\delta(\mathcal{C}):=\frac{\sum_{i=1}^{m}\Sarea(C_i)}{\Sarea(\S^2)}=\frac{2m\pi(1-\cos \rho)}{4 \pi} = \frac{1-\cos \rho}{2} m .$$
\end{Definition}

\begin{Theorem}\label{thinnest-TS-covering}
Let $\mathcal{C}$ be a TS-covering of $\S^2$ by finitely many congruent spherical caps of angular radius $0<\rho<\frac{\pi}{2}$. Then
$$\Delta({\rho}):=\frac{\pi(1-\cos\rho)}{4\arctan\left(\frac{1}{\cos\rho}\right)-\pi}\leq \delta(\mathcal{C}).$$
\end{Theorem}

\begin{figure}[ht]
\begin{center}
\includegraphics[width=0.55\textwidth]{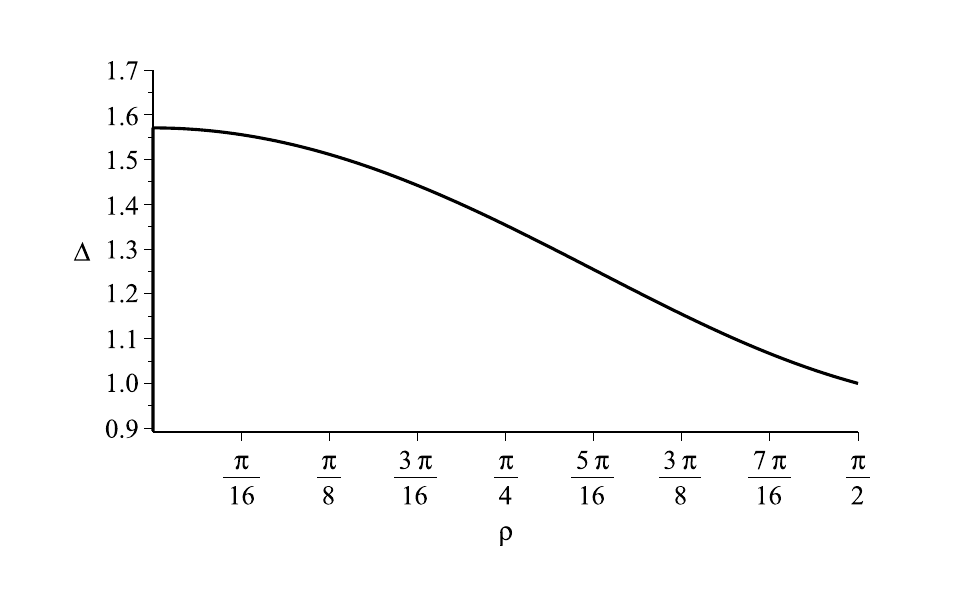}
\caption{The graph of the function $\Delta(\rho)$ defined in Theorem~\ref{thinnest-TS-covering}.}
\label{graph2}
\end{center}
\end{figure}

We note that an elementary computation, using L'Hospital's Rule, yields that $\Delta(\rho)$ is strictly decreasing over $(0, \frac{\pi}{2})$ with $\lim_{\rho \to 0^{+}} \Delta(\rho) = \frac{\pi}{2}\approx 1.57$ and $\lim_{\rho \to \frac{\pi}{2}^{-}} \Delta(\rho) = 1$ (see Figure~\ref{graph2}). We note that according to \cite{FeFe87}, $\frac{\pi}{2}$ is the smallest density of TS-coverings of $\E^{2}$ with congruent disks.

\subsection{On the analogue of Problem~\ref{separable-coverings} in high dimensions}

It would be very interesting to find the corresponding higher dimensional analogues of the above mentioned results. The following notations are natural extensions of the lower dimensional ones. Let $\E^{d}$ denote the $d$-dimensional Euclidean space, with inner product $\langle\cdot ,\cdot\rangle$ and norm $\|\cdot\|$. Its unit sphere centered at the origin $\o$ is $\S^{d-1}:= \{\x\in\E^d\ |\ \|\x\|= 1\}$. A {\it $k$-dimensional great sphere} of $\S^{d-1}$ is an intersection of $\S^{d-1}$ with a $(k+1)$-dimensional linear subspace of $\E^d$, where $0\leq k\leq d-2$.  Two points are called {\it antipodes} if they form a $0$-dimensional great sphere of $\S^{d-1}$. The set $C_{{\mathbb S}^{d-1}}[\x, \alpha]:=\{\y\in{\mathbb S}^{d-1} | \langle\x ,\y\rangle\geq \cos\alpha\}$  is called the (closed) {\it $(d-1)$-dimensional spherical cap} of angular radius $\alpha$ centered at $\x\in{\mathbb S}^{d-1}$ for $0<\alpha\leq \frac{\pi}{2}$. $C_{{\mathbb S}^{d-1}}[\x, \frac{\pi}{2}]$ is called a (closed) {\it hemisphere}.

In this section we raise and investigate the analogue of Problem~\ref{separable-coverings} in $\S^{d-1}$ for $d>3$. 

\begin{Problem}\label{high-dimensional-separable-coverings}
For given $n>1$ find the smallest $R>0$ such that there exists an arrangement of $n$ $(d-2)$-dimensional great spheres in $\S^{d-1}$ ($d>3$) with the property that each $(d-1)$-dimensional cell of the tiling of $\S^{d-1}$ generated by them can be covered by a spherical cap of radius $R$. Let us denote this $R$ by $R_{\rm gs}(n, \S^{d-1})$. 
\end{Problem}

\begin{Definition}
Let $n >1$ be fixed. An arrangement of $n$ $(d-2)$-dimensional great spheres in $\S^{d-1}$ is called \emph{$n$-extremal} if each $(d-1)$-dimensional cell of the tiling of $\S^{d-1}$ generated by them can be covered by a spherical cap of radius $R_{\rm gs}(n, \S^{d-1})$.
\end{Definition}

Clearly, $R_{\rm gs}(n, \S^{d-1})=\frac{\pi}{2}(=90^{\circ})$ for all $1<n<d$ and $d\geq 3$. On the other hand, we have

\begin{Theorem}\label{thm:highdim}\text{}
\begin{itemize}
\item[(i)] $R_{\rm gs}(d, \S^{d-1})=\arccos\frac{1}{\sqrt{d}}$ and a family of $d$ $(d-2)$-dimensional great spheres is $d$-extremal in $\S^{d-1}$, $d\geq 3$ if and only if it consists of mutually orthogonal $(d-2)$-dimensional great spheres.
\item[(ii)] $ \arcsin\left(\frac{d}{2\sum_{i=0}^{d-1}\binom{n-1}{i}}\right) < R_{\rm gs}(n, \S^{d-1}) $ holds for $n> d>3$.
\item[(iii)] $R_{\rm gs}(n, \S^{d-1}) \leq \arcsin \left( \frac{4}{\sqrt[d]{n}} \cdot \sqrt{\frac{2d-2}{d}} \right)$ holds for any $n>d\geq 3$.
\end{itemize}
\end{Theorem}

The proof of Part (i) of Theorem~\ref{thm:highdim} is based on B\"or\"oczky's theorem \cite{boroczky} stating that among simplices of $\S^{d-1}$ which are contained in a spherical cap $C$ of angular radius $<\frac{\pi}{2}$, the regular simplices inscribed in $C$ have maximal volume. We give a new and short proof of this theorem (see Lemma~\ref{lem:density} and Corollary~\ref{Boroczky}).

In the rest of our paper we prove the theorems stated.

\section{Proof of Theorem~\ref{thm:few_caps}}\label{sec:few_caps}

Let $T$ be a spherical triangle on $\S^2$ with vertices $\a, \b, \c$. Note that the three side lines (great circles containing a side of $T$) dissect $\S^2$ into eight spherical triangles. We call the three triangles of these that share a side with $T$ \emph{adjacent} triangles. The union of $T$ and one of the three adjacent triangles is a lune, which we call a \emph{lune generated by $T$}. Like for Euclidean triangles, there is a unique spherical cap $S$ inscribed in $T$, i.e., touching each side of $T$, whose radius we call the {\it inscribed radius} of $T$. Nevertheless, unlike in the Euclidean plane, the radius of this cap may not be maximal among all spherical caps contained in $T$. Indeed, while it is easy to see that a largest spherical cap in $T$ touches at least two sides of $T$, it may happen that the largest spherical cap inscribed in a lune generated by $T$ is contained in $T$. Thus, the radius of the largest cap in $T$ called the {\it inradius} of $T$ is either the inscribed radius of $T$, or the inradius of a lune generated by $T$ (i.e., the radius of the largest cap contained in that lune), and the latter happens exactly when the largest cap inscribed in one of these lunes is contained in $T$. We denote the inradius of $T$ by $\rho(T)$, and the inscribed radius of $T$ by $\rho^i(T)$. Clearly, $\rho^i(T)\leq \rho(T)$. The following statement characterizes the case of equality.

\begin{Lemma}\label{lem:inradius}
Let $T$ be a spherical triangle on $\S^2$ with vertices $\a, \b, \c$. Let the angle of $T$ at $\a$, $\b$, $\c$ be denoted by $\alpha, \beta$ and $\gamma$, respectively, and let the length of the side of $T$ opposite of $\a$, $\b$ and $\c$ be denoted by $a$, $b$, $c$, respectively. Let $T_a, T_b$ and $T_c$ denote the adjacent triangle not containing $\a, \b$ and $\c$, respectively. Assume that $a \leq b \leq c$. Then we have that
\begin{itemize}
\item[(i)] $\rho(T) \leq \frac{1}{2} \alpha = \frac{1}{2} \min \{ \alpha, \beta, \gamma \}$; 
\item[(ii)] there is at most one triangle adjacent to $T$ whose perimeter is less than $\pi$, and if there is, then it is $T_a$;
\item[(iii)] if the perimeter of $T_a$ is less than $\pi$, then $\rho(T) = \frac{\alpha}{2}$, and otherwise $\rho(T) = \rho^i(T)$.
\end{itemize}
\end{Lemma}

\begin{proof}
Note that in a spherical triangle the smallest angle is opposite of the shortest side. Furthermore, the inradii of the three lunes generated by $T$ are exactly the half angles of $T$. This yields Part (i).

Observe that the sides of $T_a$ are of length $a$, $\pi - b$, $\pi-c$. Thus, the fact that its perimeter is less than $\pi$ is equivalent to the inequality $\pi < b+c-a$. We obtain similar inequalities for $T_b$ and $T_c$. Assume that, say $\pi < b+c-a$ and $\pi < a+c-b$. Then, summing up, we obtain $2 \pi < 2 c$, which is clearly impossible. On the other hand, by our conditions, we have $ a+b-c \leq a+c-b \leq b+c-a$, and thus, if there is an adjacent triangle with perimeter less than $\pi$, it is $T_a$, implying Part (ii).

Now, we prove Part (iii). First, observe that if $T_a$ contains at least one of the midpoints of the two half great circles bounding the lune $T_a \cup T$, then, by the triangle inequality, the perimeter of $T_a$ is clearly at least $\pi$, and $\rho(T) < \frac{1}{2} \alpha$ implies $\rho(T) = \rho^i(T)$. Thus, we may restrict ourselves to the case that $T_a$ does not contain the midpoints of the half great circles bounding $T_a \cup T$.

Let $C$ denote the largest spherical cap inscribed in the lune $T_a \cup T$. Let $\b'$, $\c'$ be points on the two half great circles in the boundary of this lune, respectively, such that $C$ is inscribed in the spherical triangle with vertices $\a$, $\b'$, $\c'$. Let $T'$ be the spherical triangle with vertices  $-\a$, $\b'$, $\c'$ and let the lengths of the sides of $T'$, opposite to $-\a$, $\b'$, $\c'$ be $a', b', c'$, respectively. We show that $a'+b'+c'=\pi$. Since for any two spherically convex bodies $K \subset L \subset \S^2$ the perimeter of $K$ is strictly less than that of $L$, the above statement implies Part (iii).
Let the tangent point of $C$ on the side of $T'$ opposite to $-\a$ be $\p_a$, and let the tangent points of $C$ on the side lines of $T'$ opposite to the vertices $\b'$ and $\c'$ be $\p_b$ and $\p_c$, respectively (cf. Figure~\ref{fig:thm1proof1}). Note that by the symmetry of $C$, we have $l(\widehat{\p_a \b'} ) = l(\widehat{\p_c \b'} )$ and $l(\widehat{\p_a \c'} ) = l(\widehat{\p_b \c'} )$, implying $a' = l(\widehat{\b' \c'} ) = l(\widehat{\p_a \b'} ) + l(\widehat{\p_a \c'} ) = l(\widehat{\p_c \b'} )+l(\widehat{\p_b \c'})$. On the other hand, $l(\widehat{\p_c \b'} ) = \frac{\pi}{2} - c'$ and  $l(\widehat{\p_b \c'} ) = \frac{\pi}{2} - b'$, which, combined with the previous equalities, yields $a'+b'+c' = \pi$.
\end{proof}

\begin{figure}[h]
\begin{minipage}[t]{0.45\linewidth}
\centering
\includegraphics[width=0.35\textwidth]{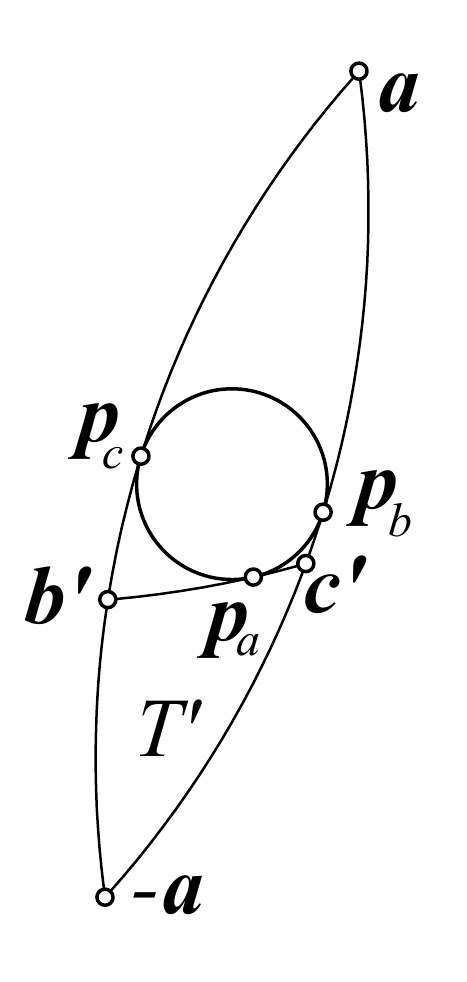}
\caption{An illustration for the proof of Lemma~\ref{lem:inradius}.}
\label{fig:thm1proof1}
\end{minipage}
\hglue0.08\textwidth
\begin{minipage}[t]{0.45\linewidth}
\centering
\includegraphics[width=0.85\textwidth]{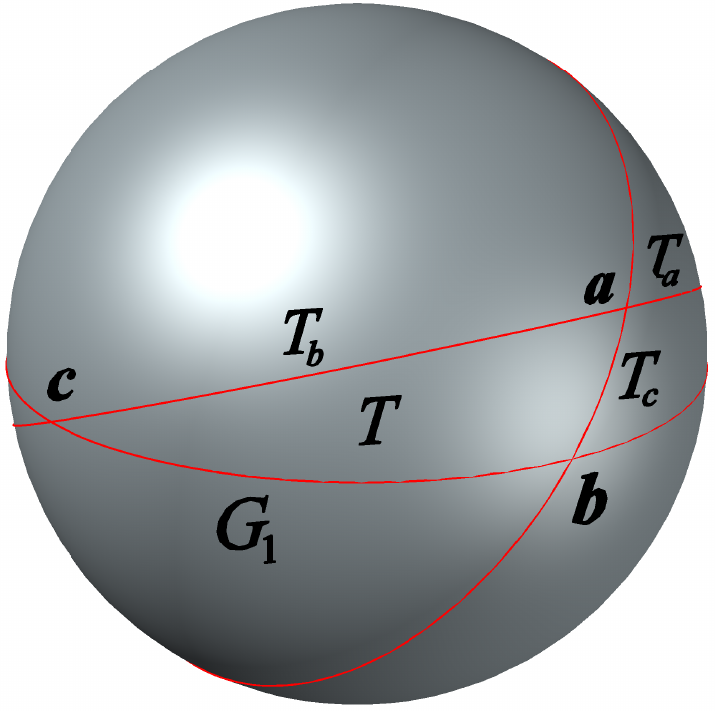}
\caption{The notation in the proof of Theorem~\ref{thm:few_caps} for $k=5$.}
\label{fig:thm1proof2}
\end{minipage}
\end{figure}

\begin{Remark}\label{rem:equality}
We note that the property that the perimeter of $T_a$ in Lemma~\ref{lem:inradius} is $\pi$ is equivalent to the fact that $\rho(T) = \rho^i(T) = \frac{\alpha}{2}$.
\end{Remark}

\begin{Remark}\label{rem:lune}
Let $T$ and $T'$ be spherical triangles such that $T'$ is adjacent to $T$, and it has the shortest perimeter among all triangles adjacent to $T$. Then
$\min \{ \rho(T), \rho(T') \} = \min \{ \rho^i(T), \rho^i(T') \}$.
\end{Remark}


Next, we turn to the proof of Theorem~\ref{thm:few_caps}. Clearly, it is sufficient to prove the first statement for the case $k=5$.
Let $\F = \{ C_1, C_2, C_3, C_4, C_5 \}$ be a TS-packing of spherical caps of radius $\rho$. Let $G_1$ be  a great circle separating two elements. Then one of the two closed hemispheres bounded by $G_1$ contains at least three elements of $\F$. Thus, without loss of generality, we may assume that $C_1$, $C_2$ and $C_3$ are contained in the same closed hemisphere. Let $G_2$ be a great circle separating $C_1$ and $C_2$. Then $G_2$ does not separate $C_3$ from $C_1$ or $C_2$, say $C_2$. Let $G_3$ denote a great circle separating $C_2$ and $C_3$. Then we obtain a closed hemisphere dissected into four spherical triangles by two half great circles such that three spherical triangles contain one element of $\F$, and the fourth one is empty. Let the vertices of the empty triangle $T$ be $\a$, $\b$ and $\c$ such that $\b, \c$ lie on $G_1$. Let the lengths of the sides of $T$ opposite to $\a$, $\b$ and $\c$ be $a, b, c$, respectively, and let the angles at these vertices be $\alpha, \beta, \gamma$, respectively. We denote the spherical triangle adjacent to $T$ and not containing $\b$ and $\c$ by $T_b$ and $T_c$, respectively. We note that the remaining fourth spherical triangle in the hemisphere, which we denote by $T_a$, is the reflected copy of the third triangle adjacent to $T$ (cf. Figure~\ref{fig:thm1proof2}).
To prove Theorem~\ref{thm:few_caps}, it is sufficient to show that
$\min \{ \rho(T_a), \rho(T_b), \rho(T_c) \} \leq \arctan \frac{3}{4}$, with equality only if $T$ is a regular triangle of side length $\arccos \frac{1}{4}$. We distinguish the following two cases.

\emph{Case 1}: $a+b+c \leq \pi$.\\
In this case, by Lemma~\ref{lem:inradius}, $\rho(T_a) = \frac{\alpha}{2}$, $\rho(T_b) = \frac{\beta}{2}$ and $\rho(T_c) = \frac{\gamma}{2}$.
Recall that the area of $T$ is $\alpha+\beta + \gamma - \pi$. Thus, by the Discrete Isoperimetric Inequality (see e.g. \cite{CsLN})
\[
\min \{ \rho(T_a), \rho(T_b), \rho(T_c) \} \leq \frac{1}{3} \left( \rho(T_a) + \rho(T_b) + \rho(T_c) \right) = \frac{\area(T)+\pi}{6}
\]
is maximal if and only if $T$ is a regular triangle of side length $\frac{\pi}{3}$. Thus, it follows that $\min \{ \rho(T_a), \rho(T_b), \rho(T_c) \} \leq \arcsin\frac{1}{\sqrt{3}}< \arctan \frac{3}{4}$.

\emph{Case 2}, $a+b+c > \pi$.\\
First, we show that in this case $\min \{ \rho(T_a), \rho(T_b), \rho(T_c) \} = \min \{ \rho^i(T_a), \rho^i(T_b), \rho^i(T_c) \}$.
This clearly holds if in each triangle the largest cap is the inscribed cap. Assume, for example, that $\rho(T_a) > \rho^i(T_a)$. Then, by the condition of Case 2, $\rho(T_a)$ is the inradius of the lune $T_a \cup T_b$ or the lune $T_a \cup T_c$. Thus, the statement follows from Remark~\ref{rem:lune}. If $\rho(T_b) > \rho^i(T_b)$ or $\rho(T_c) > \rho^i(T_c)$, we may apply a similar argument.

We recall the following formula (see e.g. \cite{Handbook}) for the radius $\rho$ of the inscribed circle of a spherical triangle with side lengths $A,B,C$:
\begin{equation}\label{eq:radius}
\tan \rho = \sqrt{\frac{\sin(P-A) \sin(P-B) \sin(P-C)}{\sin P}}, 
\end{equation}
where $P=\frac{A+B+C}{2}$. Note that for any triangle we have $0 < \rho < \frac{\pi}{2}$, and on this interval $x \mapsto \tan x$ is a positive, strictly increasing function. Thus, it is sufficient to find the maximum of
\begin{equation}\label{eq:inscribed}
F(a,b,c) = \sqrt[3]{ \tan^2 (\rho^i(T_a)) \tan^2 (\rho^i(T_b)) \tan^2 (\rho^i(T_c)) }
\end{equation}
under the condition $\pi < a+b+c < 2\pi$, $0 < a,b,c < \pi$.

Note that the side lengths of $T_a$ are $a, \pi-b, \pi-c$. Thus, applying the formula in (\ref{eq:radius}), we have that
\[
\tan \rho^i(T_a) = \sqrt{\frac{\sin p \sin(p-b) \sin(p-c)}{\sin (p-a)}},
\]
where $p=\frac{a+b+c}{2}$. Using similar formulas for $\rho^i(T_b)$ and $\rho^i(T_c)$ and substituting into (\ref{eq:inscribed}), we obtain that
\[
F(a,b,c) = \sqrt[3]{ \sin^3 p \sin (p-a) \sin(p-b)  \sin(p-c) }.
\]
An elementary computation shows that the function $x \mapsto \ln \sin x$ is strictly concave on the interval $(0,\pi)$. This yields that
\[
F(a,b,c) \leq F\left( \frac{2p}{3},\frac{2p}{3},\frac{2p}{3} \right) = \sin p \sin \frac{p}{3},
\]
with equality if and only if $a=b=c= \frac{2p}{3}$ for any fixed value of $p$. 
Finally, by elementary calculus methods, one can show that in the interval $\frac{\pi}{2} < p < \pi$, the expression $\sin p \sin \frac{p}{3}$ is maximal if and only if $\frac{2p}{3} = \arccos \frac{1}{4}$, which yields the assertion for $k=5,6$.

\begin{figure}[h]
\begin{minipage}[t]{0.44\linewidth}
\centering
\includegraphics[width=0.9\textwidth]{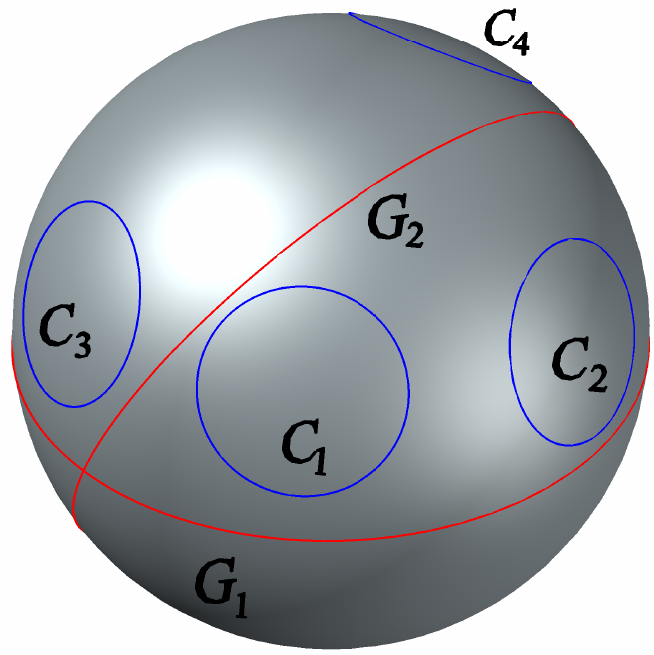}
\caption{An illustration for Case 1 in the proof of Theorem~\ref{thm:few_caps} for $k=7$.}
\label{fig:thm1proof3}
\end{minipage}
\hglue0.1\textwidth
\begin{minipage}[t]{0.44\linewidth}
\centering
\includegraphics[width=0.9\textwidth]{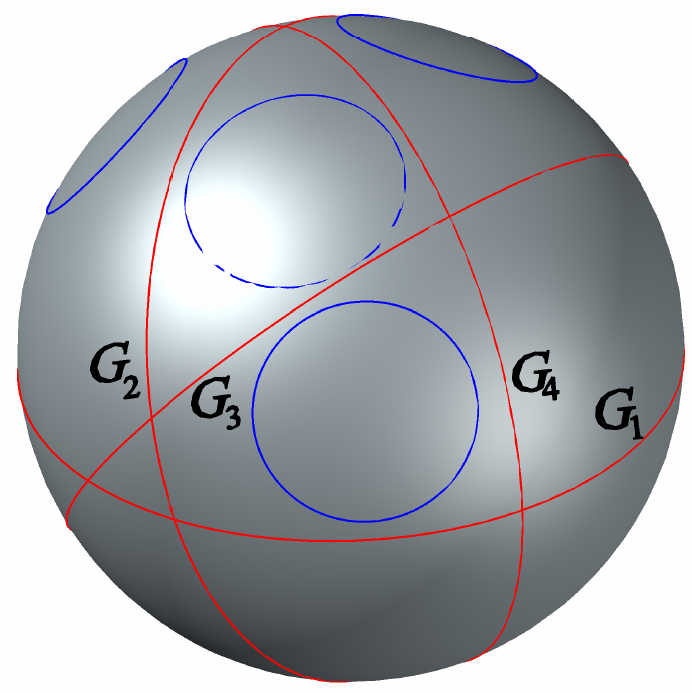}
\caption{An illustration for Case 2 in the proof of Theorem~\ref{thm:few_caps} for $k=7$.}
\label{fig:thm1proof4}
\end{minipage}
\end{figure}

Now, we prove Theorem~\ref{thm:few_caps} for $k=7,8$. Again, it is sufficient to prove it for $k=7$. Let $\F = \{ C_1, C_2, C_3, C_4, C_5, C_6, C_7 \}$ be a TS-packing of spherical caps of radius $\rho$. Let $G_1$ be  a great circle separating two elements. Then one of the two closed hemispheres bounded by $G_1$ contains at least four elements of $\F$. Thus, without loss of generality, we may assume that $C_1$, $C_2$, $C_3$ and $C_4$ are contained in the same closed hemisphere $H$ bounded by $G_1$. We distinguish two cases.

\emph{Case 1:} there is a great circle $G_2$ that separates exactly one pair of the four caps in $H$ from the other pair (cf. Figure~\ref{fig:thm1proof3}). Then $H$ is decomposed into two lunes each of which contains exactly two of the caps. In this case we can decompose each lune into two spherical triangles (or two lunes) such that each triangle (or lune) contains one of the caps. Note that then the area of at least one of the four triangles or lunes is at most $\frac{\pi}{2}$. We recall the spherical version of a theorem of Dowker, stating that among spherical polygons of given number of sides and circumscribed about a given spherical cap the regular one has minimal area (see \cite{FTF, Mo}); this yields that $\rho \leq \arcsin \frac{1}{\sqrt{3}}$ with equality if and only if the four triangles are four octants of $\S^2$, implying the second statement of Theorem~\ref{thm:few_caps} in this case.

\emph{Case 2}: any separating great circle separates one cap from the other three. Then one can check that the following holds:
\begin{enumerate}
\item[(i)] there are three half great circles in $H$, generated by three great circles $G_2$, $G_3$, $G_4$, such that they decompose $H$ into four triangles and three quadrangles such that exactly one triangle (called central triangle) is disjoint from the boundary $G_1$ of $H$,
\item[(ii)] the four spherical caps are contained in the central triangle and the three quadrangles (cf. Figure~\ref{fig:thm1proof4}).
\end{enumerate}
In this case we observe that $C_1, \ldots, C_4$, together with $-C_1, \ldots, -C_4$, form a TS-packing of eight spherical caps of radius $\rho$. On the other hand, the great circles $G_2$ and $G_3$ decompose $\S^2$ into four lunes such that each lune contains exactly two of the caps $C_1, C_2, \ldots, -C_4$. Thus, the argument in Case 1 can be applied to prove the second statement of Theorem~\ref{thm:few_caps} in this case.

\section{Spherical Moln\'ar decomposition}

This section introduces a new decomposition technique in $\S^2$, whose Euclidean analogue has been discovered by Moln\'ar \cite{Molnar}. It will be needed for the proof of Theorem~\ref{main2}.

Consider a finite point system $X = \{ \p_1, \ldots, \p_k \}$ on the sphere $\S^2$ with the property that any open hemisphere contains at least one point.
Then $\conv (X)$ contains the origin $\o$ in its interior, and hence, projecting from $\o$ the faces of the convex polyhedron $\conv (X)$ onto $\S^2$ yields a spherical mosaic with the points of $X$ as vertices. Note that every cell of this mosaic has a circumscribed circle passing through each of its vertices, which coincides with the circumcircle of the corresponding face of $\conv (X)$. Furthermore, the plane containing this face of $\conv (X)$ separates the spherical cap bounded by the circumcircle from $\conv (X)$, implying that the induced spherical mosaic has the property that the intersection of the circumdisk of any cell with $X$ is the vertex set of the cell. This spherical mosaic is called the \emph{Delaunay decomposition}, or \emph{$D$-decomposition} of $X$. Now, by means of Lemma~\ref{lem:sphericalMolnar} we define another decomposition, which we call \emph{spherical Moln\'ar decomposition} in short, $M$-decomposition.

Let $F$ be a cell of the $D$-decomposition, and let us denote the circumdisk of $F$ by $C_F$, and the center of $C_F$ by $\o_F$. Note that by our construction, the radius of $C_F$ is strictly less than $\frac{\pi}{2}$.
If $F$ does not contain $\o_F$, then there is a unique side of $F$ that separates it from $F$ in $C_F$. We call this side a \emph{separating side} of $F$. If $\widehat{\p_i\p_j}$ is a separating side of $F$, then we call the polygonal curve $\widehat{\p_i\o_F} \cup \widehat{\o_F \p_j}$ the \emph{bridge} associated to the separating side $\widehat{\p_i\p_j}$ of $F$ (cf. Figure~\ref{fig:molnar}).

\begin{figure}[ht]
\begin{center}
 \includegraphics[width=0.2\textwidth]{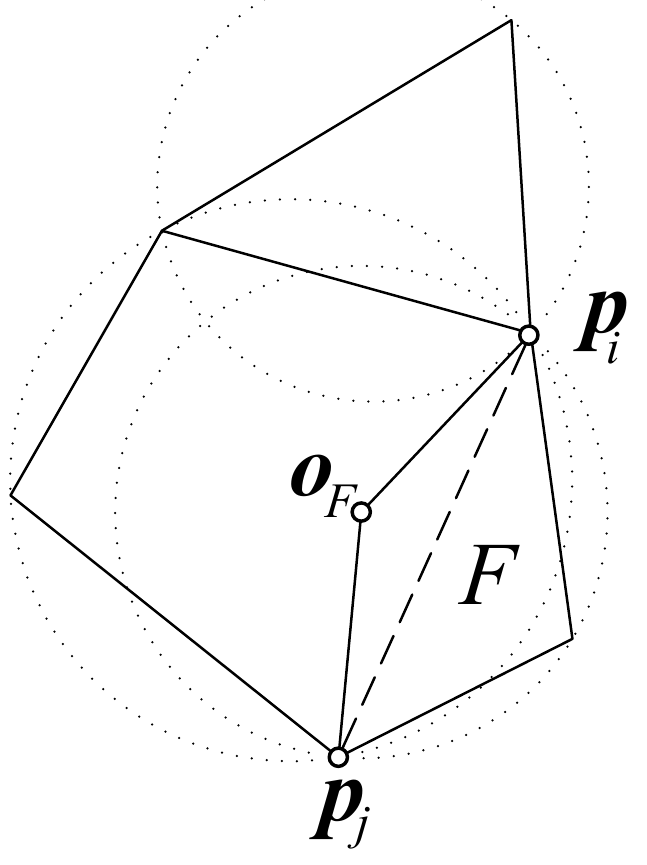}
 \caption{An illustration for the M-decomposition of a point system on $\S^2$. In the figure, great circle arcs are represented with straight line segments; the edges of the cells, the circumcircles of the cells and the separating sides are denoted by solid, dotted and dashed lines, respectively.}
\label{fig:molnar}
\end{center}
\end{figure}

Our main lemma is the following.

\begin{Lemma}\label{lem:sphericalMolnar}
If we replace all separating sides of a $D$-decomposition by the corresponding bridges, we obtain a decomposition.
\end{Lemma}

\begin{proof}
As in \cite{Molnar}, the proof is based on showing the following two statements.
\begin{itemize}
\item[(a)] Bridges may intersect only at their endpoints.
\item[(b)] A bridge may intersect a side in the $M$-decomposition only at its endpoints.
\end{itemize}

To show (a), recall that for any cell $F$, the points of $X$ closest to $\o_F$ are exactly the vertices of $F$. Thus, if $\widehat{\o_F\p_i}$ and $\widehat{\o_{F'}\p_j}$ are components of bridges with $\p_i \neq \p_j$ and $\o_F \neq \o_{F'}$, then the bisector of the spherical segment $\widehat{\p_i\p_j}$ separates $\widehat{\o_F\p_i}$ and $\widehat{\o_{F'}\p_j}$, which, since $\o_F \neq \o_{F'}$, yields that the components are disjoint.

Now, we show (b), and let $\widehat{\p_i\p_j}$ be a separating side of the cell $F$. First, observe that by the definition of separating side, the spherical triangle $T$ with vertices $\p_i$, $\p_j$ and $\o_F$ intersects finitely many cells of the $D$-decomposition at a point different from $\widehat{\p_i\p_j}$, each of which is different from $F$. Let $F'$ be the cell adjacent to $F$ and having $\widehat{\p_i\p_j}$ as a side.
Note that both $\o_F$ and $\o_{F'}$ lie on the bisector of $\widehat{\p_i\p_j}$ such that if $\m$ is the midpoint of this arc, then $\o_F \in \widehat{\o_{F'}\m}$, implying that the radius of $C_{F'}$ is strictly greater than that of $C_F$. If, apart from $\widehat{\p_i\p_j}$, $F'$ contains $T$ in its interior, then (b) clearly holds for the bridge associated to $\widehat{\p_i\p_j}$. On the other hand, if $F'$ does not contain $T$ in its interior apart from $\widehat{\p_i\p_j}$, then there is a separating side $\widehat{\p_{i'}\p_{j'}}$ of $F'$, removed during the construction, and this side is the unique side containing any point of the bridge apart from $\p_i$ and $\p_j$. Thus, no side of $F'$ in the $M$-decomposition contains an interior point of the bridge. Furthermore, observe that if any side in the $M$-decomposition contains an interior point of $\widehat{\p_i\o_F} \cup \widehat{\o_F \p_j}$, then it also contains an interior point of $\widehat{\p_{i'}\o_{F'}} \cup \widehat{\o_{F'} \p_{j'}}$. Consequently, we may repeat the argument with $F'$ playing the role of $F$, and since only finitely many cells may intersect $T$ at a point different from $\widehat{\p_i\p_j}$, and each is different from $F$, we conclude that $\widehat{\p_i\o_F} \cup \widehat{\o_F \p_j}$ may intersect any side of the $M$-decomposition only at $\p_i$ or $\p_j$.
\end{proof}

\section{Proof of Theorem~\ref{main2}}

\subsection{Proof of Part (i)}
Notice that $\delta(\F_m) \leq \delta(\rho)$ implies via a straightforward computation that $r_{\rm STam}(k, \S^2)\leq\arccos\frac{1}{\sqrt{2}\sin\left(\frac{k}{k-2}\frac{\pi}{4}\right)}$ holds for all $k\geq 5$. So, we are left to prove the inequality $\delta(\F_m) \leq \delta(\rho)$.

First, observe that based on the method of Remark~\ref{central-symmetry} we may assume that $\F_m$ is $\o$-symmetric. Furthermore, if all the centers of the caps are contained in a given great circle, then an easy computation yields the statement. Thus, we may assume that any open hemisphere contains the center of at least one cap, or in other words, the circumradius of any cell in the $D$-decomposition of the set of the centers is strictly less than $\frac{\pi}{2}$.
In the proof we use the $M$-decomposition of the set of centers of $\F_m$.
Let $\widehat{\c_i\c_j}$ be an edge of a cell $Q$ of the $D$-decomposition of the set of centers of $\F_m$, and let $C$ be the circumcircle of $Q$.
Recall that if $\widehat{\c_i\c_j}$ separates the center $\c$ of $C$ from $Q$, then the curve $\widehat{\c_i\c} \cup \widehat{\c\c_j}$ is called a bridge of the $D$-decomposition.
According to Lemma~\ref{lem:sphericalMolnar} if any separating side of the $D$-decomposition is replaced by the corresponding bridge, then we get the $M$-decomposition of the set of centers of $\F_m$.

We prove the statement for a wider class of packings. More specifically, we call a packing of spherical caps an \emph{$R$-separable} packing if any triple of caps whose centers are at pairwise distances at most $R$ is a TS-packing. In the following, we assume that $\F_m$ is a $(2R_{\rho})$-separable packing with $R_{\rho}= \arcsin (\sqrt{2} \sin \rho)$, and note that $R_{\rho} < \frac{\pi}{2}$ for any $\rho < \frac{\pi}{4}$. (Observe that for a cap of radius $R_{\rho}$ there is an inscribed equilateral quadrangle of side length $2\rho$.) In addition, we may assume that $\F_m$ is {\it $(2R_{\rho})$-saturated} in  $\S^2$, i.e., every point $\p$ of $\S^2$ is at distance at most $2R_{\rho}$ from the center of a cap in $\F_m$, as otherwise the spherical cap of center $\p$ and radius $\rho$ can be added to $\F_m$ preserving its $(2R_{\rho})$-separability. This implies that the radius of any circumdisk of the cells in the $D$-decomposition of the centers of $\F_m$ is at most $2 R_{\rho}$, and note that as the circumradius of any cell is strictly less than $\frac{\pi}{2}$, this condition is already satisfied if $R_{\rho} \geq \frac{\pi}{4}$.

Thus, if the radius of $C$ is less than $R_{\rho}$, then $Q$ is a triangle containing the center of $C$ in its interior. Suppose for contradiction that there is a bridge associated to a side $E$ of $Q$, and let $Q'$ be the adjacent cell of the $D$-decomposition, with circumcircle $C'$. Then $E$ separates $Q'$ and the center of $C'$ in $C'$, implying that the great circle through $E$ does not separate the centers of $C$ and $C'$. Since neither $C$ nor $C'$ contains the center of any cap of $\F_m$ in its interior, from this it follows that the radius of $C'$ is strictly less than that of $C$; that is, it is strictly less than $R_{\rho}$ implying that $Q'$ is a triangle containing the center of $C'$ in its interior, a contradiction. Thus, if the radius of $C$ is less than $R_{\rho}$, then every side of $Q$ is a side in the $M$-decomposition of $\F_m$.

We further decompose the remaining cells of the $M$-decomposition by the spherical segments connecting the center of the circumcircle to the vertices of the corresponding Delaunay cell. We call the so obtained decomposition the \emph{refined $M$-decomposition}, which contains two types of cells:
\begin{itemize}
\item[(a)] triangles with side length at least $2\rho$ which contain their circumcenters in their interiors and their circumradii are less than $R_{\rho}$;
\item[(b)] the closure of the difference of two isosceles triangles with legs of length at least $R_{\rho}$ and common base of length at least $2 \rho$.
\end{itemize}

Let $P$ be a cell of the refined $M$-decomposition and let $\varphi$ denote the sum of the (internal) angles of $P$ at the vertices that are centers of elements of $\mathcal{F}_m$. We define the density of $\mathcal{F}_m$ in $P$ as
\begin{equation}\label{eq:celldensity}
\rho_P(\mathcal{F}_m) = \frac{\varphi(1-\cos \rho) }{\Sarea \left( P \right)}.
\end{equation}

In order to prove the inequality $\delta(\F_m) \leq \delta(\rho)$ for any $(2R_{\rho})$-separable packing $\F_m$ (which is also $(2R_{\rho})$-saturated in $\S^2$),
we prove the following two lemmas from which the inequality at hand follows. Indeed, Lemma~\ref{lem:isosceles} applies to cells satisfying (b), and Lemma~\ref{lem:acute} applies to cells satisfying (a), where density in a cell is computed as in (\ref{eq:celldensity}).

\begin{Lemma}\label{lem:isosceles}
Let $Q$ be an isosceles spherical triangle, with vertices $\a_1, \a_2, \b$, where $\b$ is the apex. Let the length of the legs of $Q$ be $b$, and that of the base be $a$, where $2\rho\leq a < \pi$ and $\frac{a}{2}< b < \frac{\pi}{2}$ . Furthermore, let the angles on the base be $\alpha$, and the angle at the apex be $\beta$. For $i=1,2$, let $S_i$ denote the spherical cap of radius $\rho$, centered at $\a_i$. Let
\[
f(a,b) =
\frac{2\alpha (1- \cos \rho)}{2 \alpha + \beta - \pi}
\]
denote the density of the packing $\{ S_1, S_2 \}$ in $Q$ as a function of $a$ and $b$. Then $f(a,b)$ is a strictly decreasing function of both $a$ and $b$.
\end{Lemma}

\begin{proof}
Let $b_1 < b_2$, $\m$ be the midpoint of $\widehat{\a_1 \a_2}$ (cf. Figure~\ref{fig:thm2proof1}), and $\b_1$ and $\b_2$ be two points on the bisector of $\widehat{\a_1\a_2}$ such that the spherical distance of $\b_i$ and $\a_1$ is $b_i$. Let $Q_i$ be the triangle with vertices $\a_1, \m, \b_i$. Note that
\[
f(a,b_i) = \frac{\Sarea \left(Q_i \cap S_1 \right)}{\Sarea \left( Q_i \right) } .
\]
Let $S$ denote the spherical cap centered at $\a_1$ and of radius $b_1$. Then we have
\[
\frac{\Sarea \left(Q_1 \cap S_1 \right)}{\Sarea \left( Q_1 \right) } > \frac{\Sarea (S_1)}{\Sarea(S)} > \frac{\Sarea \left(Q_2 \cap S_1 \right)}{\Sarea \left( Q_2 \right) },
\]
which yields that $f(a,b)$ is a strictly decreasing function of $b$.

We show that $f(a,b)$ is a strictly decreasing function of $a$. First, we note that for any fixed value of $b$, $\alpha$ is a strictly decreasing function of $a$, which implies that it is sufficient to show that the function $g(\alpha,b) = \frac{\Sarea \left(Q \cap \left( S_1 \cup S_2 \right) \right)}{\Sarea \left( Q \right) }$ of $\alpha$ and $b$ is a strictly increasing function of $\alpha$. Note that by the dual spherical law of cosines applied for the spherical triangle with vertices $\a_1, \m, \b$, we have
\[
0 = - \cos \alpha \cos \frac{\beta}{2} + \sin \alpha \sin \frac{\beta}{2} \cos b,
\]
implying that $\tan \left( \frac{\pi}{2}- \frac{\beta}{2} \right) = \cot \frac{\beta}{2} = \cos b \tan \alpha$. Thus,
\[
g(\alpha,b) = \frac{\alpha (1- \cos \rho)}{2\alpha + \beta - \pi} = \frac{\frac{1}{2}(1- \cos \rho)}{1 - \frac{1}{\alpha}\left( \frac{\pi}{2} - \frac{\beta}{2} \right) } = \frac{\frac{1}{2}(1- \cos \rho)}{ 1 - \frac{\arctan (\cos b \tan \alpha)}{\alpha} }.
\]
Since the expression $\frac{\arctan (\cos b \tan \alpha)}{\alpha}$ is the slope of the chord of the graph of the function $\alpha \mapsto \arctan (\cos b \tan \alpha)$ connecting the points of the graph at $0$ and $\alpha$, to prove the monotonicity property of $g$ it is sufficient to prove that this function is strictly convex. Hence, the assertion follows from the fact that
\[
\partial_{\alpha}^2 \arctan (\cos b \tan \alpha) = \frac{2 \sin^2 b \cos b (1+ \tan^2 \alpha)}{(1+ \cos^2 b \tan^2 \alpha)^2} > 0
\]
if $0 < b < \frac{\pi}{2}$.
\end{proof}

\begin{figure}[h]
\begin{minipage}[t]{0.35\linewidth}
\centering
\includegraphics[width=0.6\textwidth]{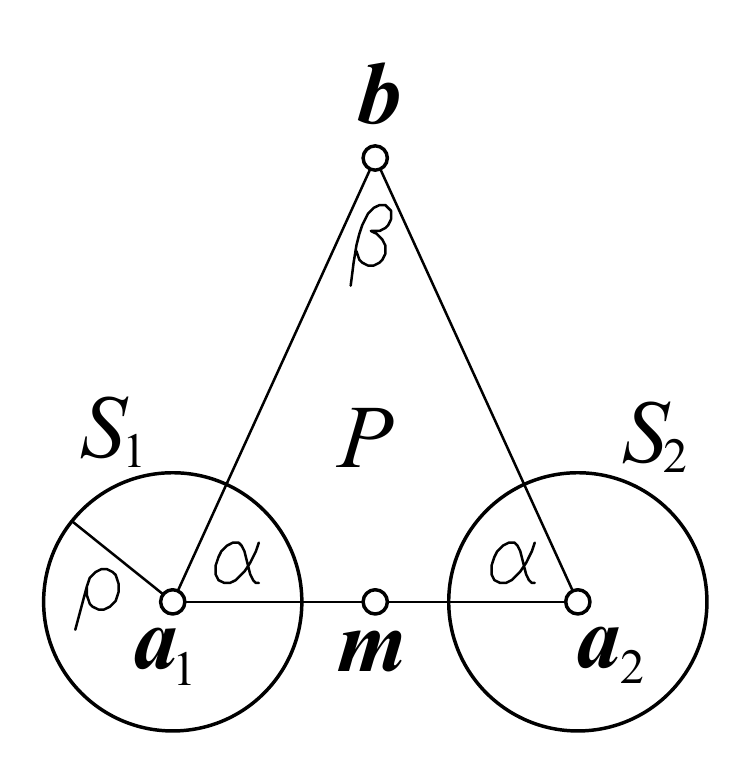}
\caption{An illustration for the proof of Lemma~\ref{lem:isosceles}.}
\label{fig:thm2proof1}
\end{minipage}
\hglue1cm
\begin{minipage}[t]{0.52\linewidth}
\centering
\includegraphics[width=0.85\textwidth]{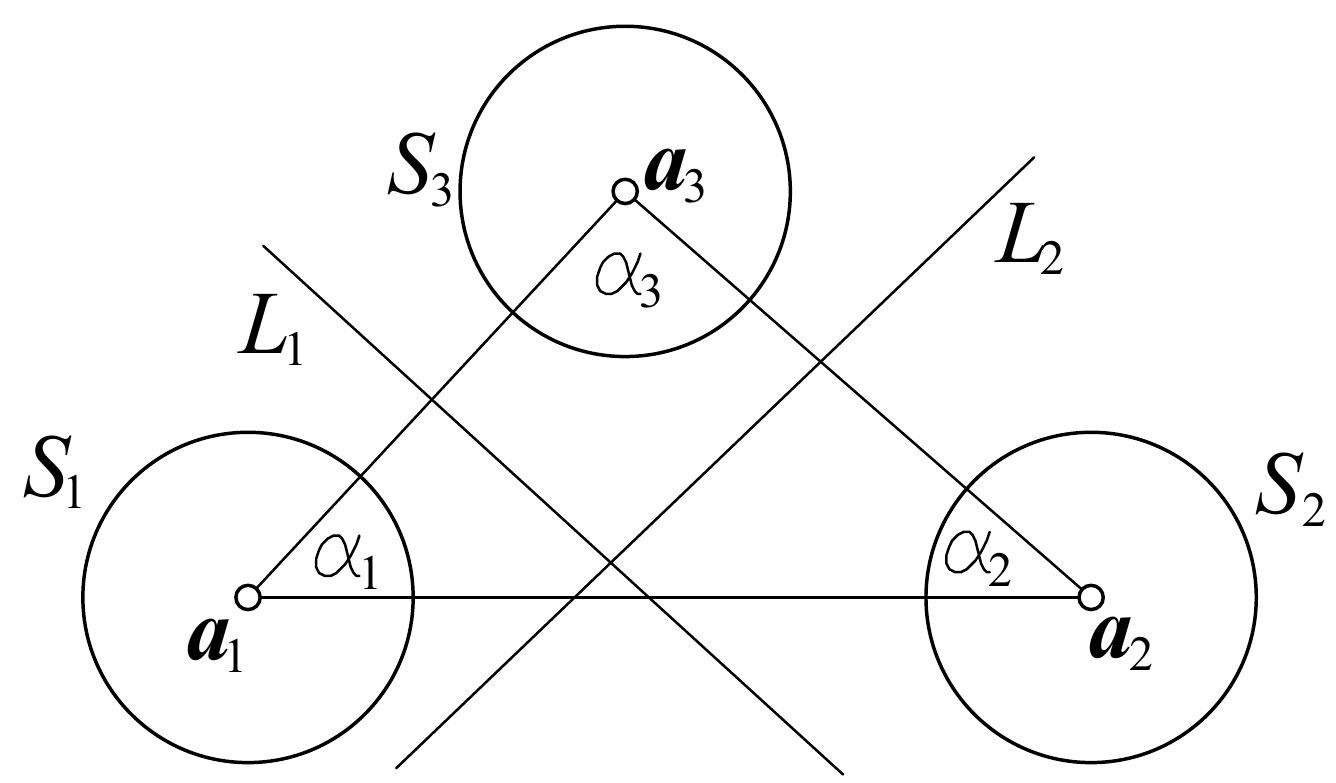}
\caption{An illustration for the proof of Lemma~\ref{lem:acute}.}
\label{fig:thm2proof2}
\end{minipage}
\end{figure}

Based on the classification of the cells of the refined $M$-decomposition above, in Lemma~\ref{lem:acute}, for brevity, we call a spherical triangle $Q$ with side lengths at least $2 \rho$, containing its circumcenter in its interior, and with circumradius less than $R_{\rho}$, a \emph{type (a)} triangle, if the spherical disks of radius $\rho$ centered at the vertices of $Q$ are members of a TS-family.

\begin{Lemma}\label{lem:acute}
Let $Q$ be a type (a) spherical triangle. Then the density computed in $Q$ of the spherical caps of radius $\rho$ centered at the vertices of $Q$ is less than $\delta(\rho)$.
\end{Lemma}

\begin{proof}
By compactness, the density of a type (a) triangle is either maximal for a type (a) triangle, or it can be upper bounded by the density of a triangle whose circumcenter lies on its boundary. Note that in the latter case we can decompose it into two triangles satisfying the conditions in Lemma~\ref{lem:isosceles}, and hence, the assertion follows from Lemma~\ref{lem:isosceles}. Thus, to prove Lemma~\ref{lem:acute} it is sufficient to show that the density of a type (a) triangle does not attain its maximum in this family, i.e., for any type (a) triangle there is another type (a) triangle with greater density.

Let $\p_i$, $i=1,2,3$ denote the vertices of $Q$, and let $S_i$ denote the spherical cap of radius $\rho$ centered at $\p_i$. Since $\{ S_1, S_2 S_3\}$ is a TS-family, we may assume, without loss of generality, that there is a great circle $L_1$ separating $S_1$ from $S_2$ and $S_3$, and another great circle $L_2$ separating $S_2$ from $S_1$ and $S_3$ (cf. Figure~\ref{fig:thm2proof2}).

Let the angle of $Q$ at $\p_i$ be $\alpha_i$. Then the density of the arrangement in $Q$ is
\[
\delta(Q) = \frac{(\alpha_1+\alpha_2+\alpha_3)(1-\cos \rho)}{\alpha_1+\alpha_2+\alpha_3 - \pi} = \frac{1- \cos \rho}{1 - \frac{\pi}{\Sarea(Q)}}.
\]
Thus, to prove the statement we need to modify the arrangement to decrease the area of $Q$. Note that with two vertices of $Q$ fixed, area remains the same if the third vertex of $Q$ moves on the Lexell circle of the two fixed vertices. Thus, if a cap $S_i$ does not touch any of $L_1$ and $L_2$, then $\p_i$ can be moved to decrease $\Sarea(Q)$.

Assume that a separating line, say $L_1$, is not touched by any of the spherical caps. If there is a cap not touching $L_2$, then we may apply the previous argument. Thus, we may assume that $L_2$ is touched by all three caps. This implies that the distance of $\p_1$ and $\p_3$ from $L_1$ is $\rho$. Observe that if we move $\p_1$ towards $\p_3$ while keeping its distance from $L_2$ fixed, we obtain a spherical triangle $Q'$ strictly contained in $Q$, yielding $\Sarea(Q) > \Sarea(Q')$. Furthermore, if both $L_1$ and $L_2$ are touched by one of the $S_i$s but one of them can be slightly moved not to touch any of the $S_i$s, we may apply a similar argument. Thus, we may assume that both $L_1$ and $L_2$ satisfy one of the following:
\begin{enumerate}
\item[(i)] it is touched by exactly two of the $S_i$s at the same point;
\item[(ii)] it is touched by all the three caps.
\end{enumerate}

First, consider the case that both $L_1$ and $L_2$ satisfy (i). Then $S_3$ touches both $L_1$ and $L_2$, $S_1$ touches $L_1$ and $S_2$ touches $L_2$, which implies that $Q$ is an isosceles triangle whose legs are of length $2 \rho$. Let $\beta$ denote the angle at $\p_3$. In the same way as in Lemma~\ref{lem:isosceles}, we obtain that $\Sarea(Q) = \beta - 2 \arctan \left( \tan \frac{\beta}{2} \cos b \right)$. The derivative of this quantity is zero only if $\tan^2 \frac{\beta}{2} \cos b = 1$, from which we have that the angle of $Q$ at $\p_1$ and $\p_2$ is $\frac{\beta}{2}$. This clearly implies that the circumcenter of $Q$ is the midpoint of $\widehat{\p_1 \p_2}$, and thus, if $Q$ satisfies our conditions we can decrease its area.

Next, consider the case that $L_1$ satisfies (i) and $L_2$ satisfies (ii). Then it is easy to see that $L_1$ and $L_2$ are perpendicular, and $L_2$ can be slightly modified not to touch any of the caps. Thus, we have that every cap touches every separating line.
Let $\c$ and $-\c$ denote the two intersection points of $L_1$ and $L_2$. Since any side of $Q$ is at most $2 R_{\rho} < \pi$, we may assume that $\p_1$ and $\p_2$ are both closer to $\c$ than to $-\c$. Now, we drop the condition that the circumcenter is contained in $T$, and show that even in this more general setting we have that $\Sarea(Q)$ is minimal if $L_1$ and $L_2$ are perpendicular, and $\p_3$ is closer to $\c$ than to $-\c$. First, if $\p_3$ is closer to $-\c$, reflecting about the plane through $\o$ and perpendicular to $[\c,-\c]$ we obtain a triangle of smaller area, and hence, we assume that the distance of $\p_3$ and $\c$, which we denote by $m$, is less than $\frac{\pi}{2}$. Let $0 < a < \frac{\pi}{2}$ denote the distance of $\p_1$ and $\c$, and let $0 < \alpha < \frac{\pi}{2}$ denote the half angle of the lune containing $S_1$. Then, if $A = \Sarea (Q)$, then
\[
\tan \frac{A}{4} = 2 \tan \frac{a}{2} \tan \frac{m}{2} = 2 \cdot \frac{\sin a}{\cos a + 1} \cdot \frac{\sin m}{ \cos m +1}.
\]
Using the spherical law of sines and the notation $X= \sin \rho$ and $Y = \sin \alpha$, we obtain that
\[
\tan \frac{A}{4} = \frac{2 X^2}{(\sqrt{Y^2-X^2}+Y)(\sqrt{1-Y^2-X^2}+\sqrt{1-Y^2})}
\]
\[ 
= 2 (Y - \sqrt{Y^2-X^2})(\sqrt{1-Y^2}-\sqrt{1-Y^2-X^2}).
\]
An elementary computation shows that the function $x \mapsto \sqrt{x}-\sqrt{x-a}$ is strictly log-convex for $x > a$. This yields that
\[
\tan \frac{A}{4} \geq 2\left( \frac{1}{2} - \sqrt{\frac{1}{2}-X^2} \right)^2,
\]
with equality if and only if $Y=\frac{1}{\sqrt{2}}$, that is, if $\alpha = \frac{\pi}{4}$ and $L_1$ and $L_2$ are perpendicular.
\end{proof}

\subsection{Proof of Part (ii)}\label{spherical-construction}

Let $L^h$ be the lune in $\S^2$ whose two vertices are the points $(0,1,0)$ and $(0,-1,0)$ such that the width of $L^h$ is $2 \alpha$ for some given $0 < \alpha < \frac{\pi}{2}$, and its bisector is contained in the closed half plane defined by the equations $z=0$ and $x \geq 0$. Let us dissect $L^h$ into $k$ mutually congruent lunes of width $\frac{2\alpha}{k}$ by half great circles with $(0,1,0)$ and $(0,-1,0)$ as their endpoints. We label these lunes by $L^h_1, L^h_2, \ldots, L^h_k$. For $i=1,2,\ldots,k$, let $L^v_i$ denote the rotated copy of $L^h_i$ by $\frac{\pi}{2}$ around the $x$-axis (cf. Figure~\ref{fig:remark4}). We denote the spherical quadrangle $L^h_i \cap L^v_j$ by $Q_{ij}$. Let $\varepsilon > 0$ be an arbitrary fixed value. We intend to show that if $k$ is sufficiently large, the inradius of $Q_{ij}$ is at least $(1-\varepsilon) \rho(\alpha, k)$, where  $\rho(\alpha, k) = \frac{\alpha \cos \alpha}{k}$. We do it by showing that if $k$ is large, then any point of $Q_{ij}$ lying on the bisector of $L^h_i$ is at spherical distance at least $(1-\varepsilon) \rho(\alpha, k)$ from the sides of $L^h_i$. Indeed, from this fact and by symmetry we obtain that for large $k$ any point of $Q_{ij}$ lying on the bisector of $L^v_j$ is at distance at least $(1-\varepsilon) \rho(\alpha, k)$ from the sides of $L^v_j$, which yields that the intersection point of the bisectors of $L^h_i$ and $L^v_j$ is at distance at least $(1-\varepsilon) \rho(\alpha, k)$ from any side of $Q_{ij}$.

Let $\p$ be any point of $Q_{ij}$ on the bisector of $L^h_i$. Let $L^v$ denote the rotated copy of $L^h$ by $\frac{\pi}{2}$ around the $x$-axis and $L$ be the bisector of $L^v$. Note that the distance of $\p$ from $L$ is at most $\alpha$. Furthermore, the distance of $\p$ from the two sides of $L^h_i$ is a strictly decreasing function of the distance of $\p$ from $L$. Thus, it is sufficient to show that if $\p_0$ is the point on the bisector of $L^h_i$ at distance $\frac{\pi}{2}-\alpha$ from the vertex $\c=(0, 1,0)$ of $L^h_i$, then
the distance of $\p_0$ from the two sides of $L^h_i$ is at least $(1-\varepsilon) \rho(\alpha, k)$. Let $\q$ be the orthogonal projection of $\p_0$ onto one of the sides of $L^h_i$ (cf. Figure~\ref{fig:remark4_part2}), and let $\rho$ denote the length of $\widehat{\p_0\q}$. Then $\p_0$, $\q$ and $\c$ are the vertices of a spherical right triangle with the right angle at $\q$, and thus, the spherical law of sines yields $\sin (\rho) = \sin \frac{\alpha}{k} \cos \alpha$. On the other hand, as $\lim_{x \to 0^{+}} \frac{\sin x}{x} = 1$, we have that if $k$ is sufficiently large, then $\sin (\rho) > \sin \left( \frac{(1-\varepsilon) \alpha}{k} \cos \alpha \right)$, implying $r_{\rm STam}(2k^2, \S^2) \geq  (1-\varepsilon) \rho(\alpha,k)$ for any fixed value of $\alpha$ and $\varepsilon$. Thus, our estimate for all sufficiently large values of $m$ follows from the observation that $\sqrt{2} \cdot \frac{49}{180} \pi \cdot \cos (49^{\circ}) > 0.793$.

\begin{figure}[h]
\begin{minipage}[t]{0.48\linewidth}
\centering
\includegraphics[width=0.8\textwidth]{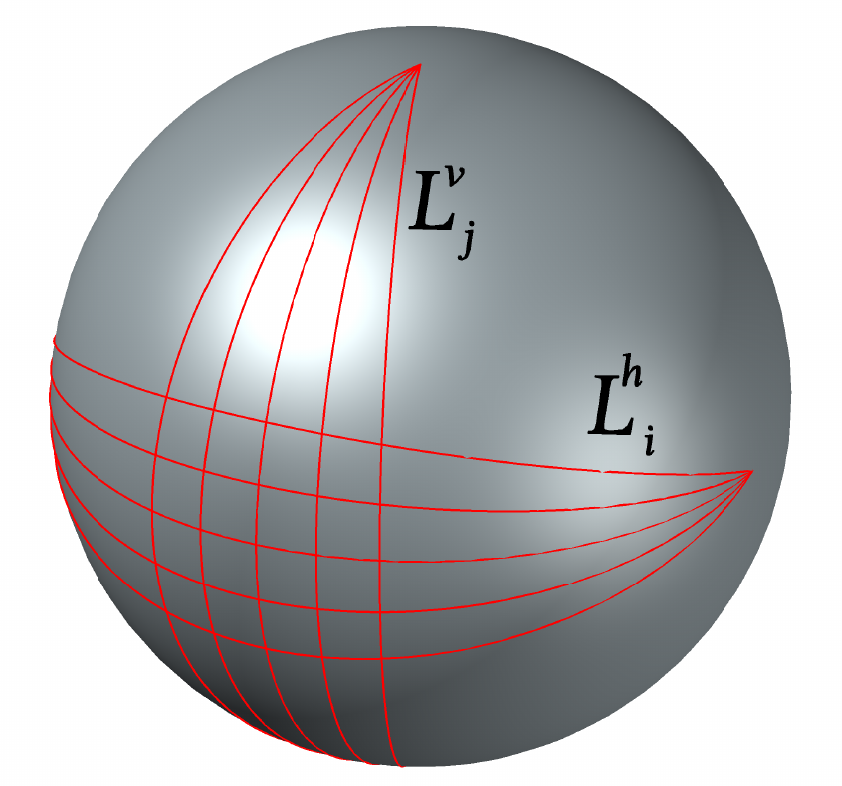}
\caption{The arrangement of lunes constructed in Section \ref{spherical-construction}.}
\label{fig:remark4}
\end{minipage}
\hglue1cm
\begin{minipage}[t]{0.48\linewidth}
\centering
\includegraphics[width=0.8\textwidth]{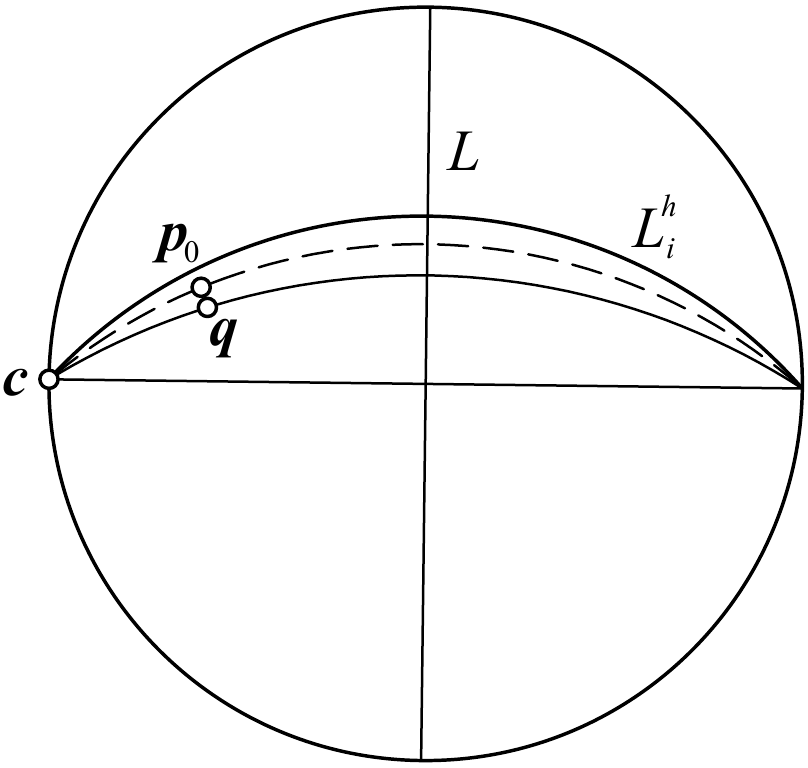}
\caption{The front view of the arrangement of lunes in Section \ref{spherical-construction}.}
\label{fig:remark4_part2}
\end{minipage}
\end{figure}

\bigskip

\section{Proof of Theorem~\ref{sub-main2}}

According to Remark~\ref{equivalence}, $\tau_{\rm{csep}}(\E^{3})$ is equal to the maximum number of spherical caps of angular radius $\frac{\pi}{6}$ that form a TS-packing in $\S^2$. Hence, using Remark~\ref{central-symmetry}, the octahedral TS-packing of Figure 1, and Corollary~\ref{upper bound for the separable kissing number}, one obtains in a straightforward way that either $\tau_{\rm{csep}}(\E^{3})=8$ (which can be achieved by a TS-packing of $8$ spherical caps of angular radius $\frac{\pi}{6}$ possessing octahedral symmetry in $\S^2$) or $\tau_{\rm{csep}}(\E^{3})=10$.

For contradiction, suppose that there is a totally separable family $\mathcal{F}$ of $10$ spherical caps of angular radius $\frac{\pi}{6}$ on $\S^2$.
We may assume that the family consists of pairs of spherical caps symmetric to $\o$, i.e., $\mathcal{F} = \{ C_1, C_2, \ldots, C_5, -C_1, -C_2, \ldots, -C_5 \}$.

First, we show that there are two great circles $G_1$ and $G_2$ that dissect $\S^2$ into four lunes each of which contains $2$ or $3$ elements of $\mathcal{F}$.
Consider any separating great circle $G$, and, without loss of generality, assume that one of the closed hemispheres bounded by $G$ contains $C_1, C_2, \ldots, C_5$. Let $G_1$ be a great circle separating $C_1$ and $C_2$. If $G_1$ separates $C_1$ at least one and at most two of the spherical caps $C_3, C_4$ and $C_5$, we are done. Thus, after relabeling the spherical caps if necessary, we have that $G_1$ separates $C_1$ from $C_2, C_3, C_4, C_5$. Now, let the great circle $G_2$ separate $C_2$ from $C_3$. Similarly like for $G_1$, we may assume that $G_2$ separates $C_2$ from $C_1, C_3, C_4, C_5$. Then $G_1$ and $G_2$ dissect $\S^2$ into four lunes satisfying the required conditions.

Now, let $G_1$ and $G_2$ be two separating great circles satisfying the condition in the previous paragraph, i.e., we assume that one of the lunes generated by them contains two elements of $\mathcal{F}$ say, $C_1$ and $C_2$, and another lune generated by them contains three elements of $\mathcal{F}$ say, $C_3, C_4$ and $C_5$. We denote the first lune by $L$ and the second lune by $L'$. Let the angle of $L$ be $\alpha$. This implies that the angle of $L'$ is $\beta = \pi-\alpha$.
In the remaining part of this section we show that $\alpha\geq \alpha_0$ and $\beta\geq\beta_0$ with $\alpha_0+\beta_0>\pi$, finishing the proof of Theorem~\ref{sub-main2} by contradiction.

First, we examine $L$. By Lemma~\ref{lem:thickness}, the value of $\alpha$ attains its minimal value $\alpha_0$ if and only if $C_1$ and $C_2$ touch each other, and both of them touch both sides of $L$ from inside. In this case $\frac{\pi}{6}$ is the inradius of a spherical triangle of side lengths $\frac{\pi}{2}, \frac{\pi}{2}$ and $\alpha_0$. Thus, (\ref{eq:radius}) yields that $\tan \frac{\pi}{6} = \sin \frac{\alpha_0}{2}$, implying that
\[
\alpha \geq \alpha_0 = 2 \arcsin \frac{1}{\sqrt{3}} > 2 \arcsin \frac{1}{2} = \frac{\pi}{3} .
\]

Next, we consider $L'$. Let $S_1$ and $S_2$ be the two sides of $L'$, and consider the system $\mathcal{S} = \{ S_1, S_2, C_3, C_4, C_5 \}$. Let $T$ be the set of the \emph{touching points} of $\mathcal{S}$, i.e., the set consisting of the points where two spherical caps, or a spherical cap and a side of $L'$ touch. We say that $\mathcal{S}$ is \emph{stable} if
\begin{enumerate}
\item[(a)] the midpoint of each side $S_i$ of $L'$ is contained in the spherical convex hull of $S_i \cap T$,
\item[(b)] each spherical cap $C_j$ is either disjoint from $T$, or its center is contained in the spherical convex hull of $C_j \cap T$.
\end{enumerate}
We remark that if $\mathcal{S}$ is not stable, then we can modify it to obtain a lune with smaller angle that contains $3$ non-overlapping spherical caps of radius $\frac{\pi}{6}$, or a stable system in a lune with the same angle. Indeed, if the midpoint of $S_i$ is not contained in the spherical convex hull of $S_i \cap T$, then we may slightly rotate $S_i$ to obtain a lune of smaller area containing all the $C_j$, where we recall that the area of a lune is proportional to its angle. On the other hand, if a cap $C_j$ intersects $T$, but its center is not contained in the spherical convex hull of $C_j \cap T$, then we may move $C_j$ slightly to obtain a spherical cap in $L'$ of radius $\frac{\pi}{6}$ that does not touch any other element of the system. If the modified system does not satisfies (a), we can slightly rotate a side of $L'$, and if it does not satisfies (b), we can slightly move another cap. Thus, the minimum of $\beta$, which we denote by $\beta_0$, is attained by a stable system.

Now, we show that if $\mathcal{S}$ is stable, then, by a suitable choice of indices, we have that
\begin{enumerate}
\item[(1)] $C_3$ and $C_4$ touch each other and each of them contains a midpoint of a side of $L'$; or
\item[(2)] $C_3$ and $C_5$ contain a midpoint of a side of $L'$, and $C_4$ touches $C_3$ and $C_5$ at opposite points on its boundary; or
\item[(3)] $C_3$ and $C_5$ touch both sides of $L'$, and $C_4$ touches $C_3$, $C_5$ and $S_2$; or
\item[(4)] $C_3$ and $C_5$ touch both sides of $L'$, and $C_4$ touches $C_3$ and $C_5$ at opposite points on its boundary (see Figure~\ref{stable}).
\end{enumerate}

\begin{figure}[ht]
\begin{center}
\includegraphics[width=0.7\textwidth]{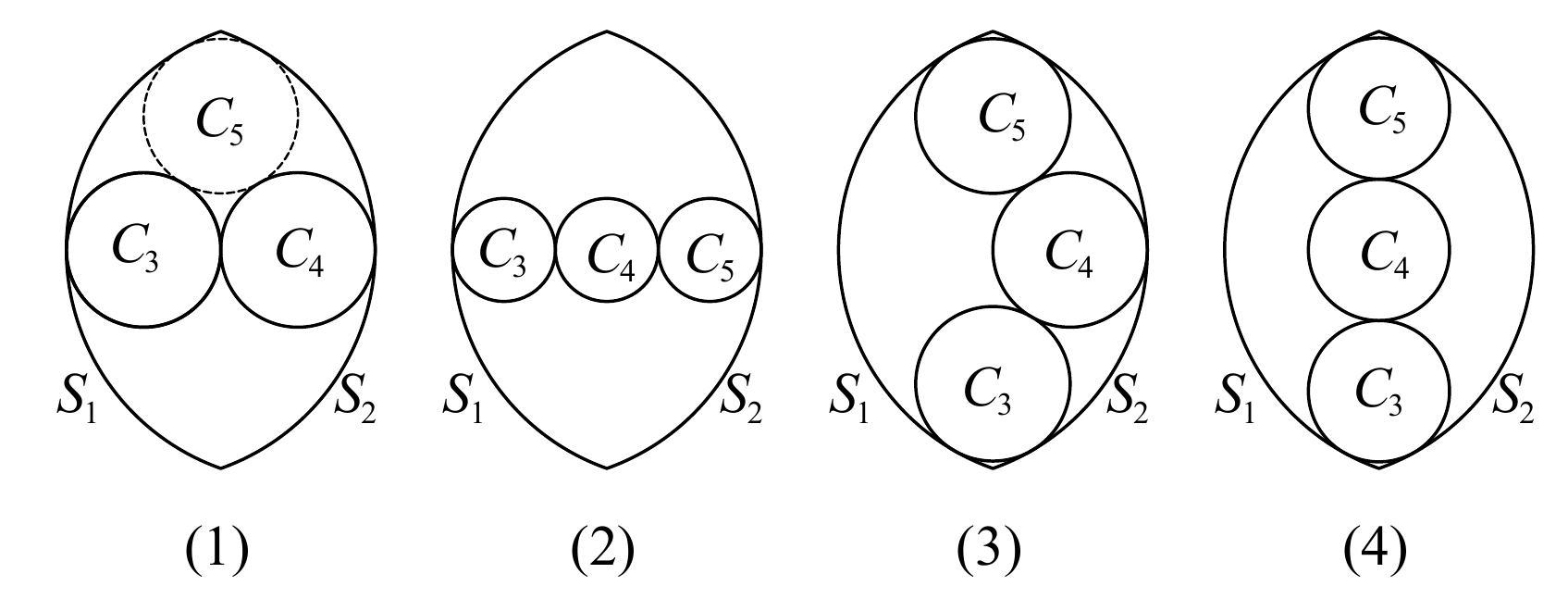}
\caption{The four types of stable systems.}
\label{stable}
\end{center}
\end{figure}

Assume that $\mathcal{S}$ is stable, and first, consider the case when there is a spherical cap say, $C_5$, that is disjoint from $T$. Note that $C_3$ and $C_4$ touch each side of $L'$ in at most one point. If the midpoints of $S_1$ and $S_2$ do not belong to $T$, then $C_3$ and $C_4$ must touch both $S_1$ and $S_2$ and each other and therefore $\beta=\alpha_0 < \frac{2\pi}{3} = 4 \cdot \frac{\pi}{6}$, implying that $L'$ does not contain a third non-overlapping spherical cap of radius $\frac{\pi}{6}$. If a midpoint of a side of $L'$ belongs to $T$, then by the stability condition, both midpoints must belong to $T$, and we have (1). Finally, using the computation on (3) at the end of this section, it follows that $C_5$ must be tangent to $C_3$ and $C_4$ and must touch both $S_1$ and $S_2$, a contradiction.

Next, consider the case when there is no spherical cap disjoint from $T$. Assume that the midpoints of both sides of $L'$ belong to $T$, say $C_3$ and $C_5$ contain these points. Then, by (b), the point of $C_3$  (resp., $C_5$) opposite to the midpoint of $S_i$ must belong to $T$, which implies that either $C_3$ and $C_5$ touch each other, or $C_4$ touches both, implying (1) and (2), respectively. If exactly one midpoint is in $T$ or if no midpoint is in $T$, a similar elementary consideration shows that (3) or (4) is satisfied, respectively.

Finally, we compute the angles of the lunes in the stable systems satisfying (1)-(4); we denote this angle of a system satisfying (i) by $\beta_i$.  Clearly, $\beta_1 = \frac{2\pi}{3}$ and $\beta_2 =\beta_4= \pi$. So, we are left with computing $\beta_3$. For $i=3,4,5$, let $\p_i$ denote the center of $C_i$ and let $\q_i$ denote the point where $C_i$ touches $S_2$. Furthermore, let $\x$ denote
the vertex of $L'$ closer to $\p_3$. Using a spherical coordinate system on $\S^2$ with the side $S_2$ being on the Equator, one can easily compute that $l(\widehat{\q_3\q_4}) =2 \arcsin \tan \frac{\pi}{6} = 2 \arcsin \frac{1}{\sqrt{3}}$ and so, $l(\widehat{\q_3 \x}) = \frac{\pi}{2} - 2 \arcsin \frac{1}{\sqrt{3}}$. Thus, applying the law of sines and algebraic transformations, we obtain that $\beta_3 = \frac{2\pi}{3}$. This implies that
\[
\beta \geq \beta_0= \min \{ \beta_1, \beta_2, \beta_3, \beta_4 \} = \frac{2\pi}{3},
\]
and therefore $\alpha_0+\beta_0>\pi$, a contradiction. This completes the proof of Theorem~\ref{sub-main2}.

\bigskip

\section{Proof of Theorem~\ref{main3}}

\subsection{Proof of Part (i)}  \label{subsec:part1}

If $n=2$, then $G_1$ and $G_2$ divides $\S^2$ into four lunes, and there is a lune with half angle at most $\frac{\pi}{4}$. Thus, the assertion follows from the fact that the inradius of a lune is equal to its half angle.

\subsection{Proof of Part (ii)}  \label{subsec:part2}

If $n=3$ and the three great circles meet at the same pair of antipodes, then by the same observation as above, there is a cell with inradius at most $\frac{\pi}{6}<\arcsin \frac{1}{\sqrt{3}}$. On the other hand, if the great circles are in nondegenerate position, then the statement follows from Theorem~\ref{thm:few_caps}.

\subsection{Proof of Part (iii)}    \label{subsec:part3}

In the following we consider the case that $n=4$. Then, depending on the degeneracy of the positions of the great circles, we have one of the following three cases.

\emph{Case 1:} all of $G_1, \ldots, G_4$ meet at the same pair of antipodes.

In this case we may apply our previous observation and conclude that one of the lunes defined by the great circles has inradius at most $\frac{\pi}{8} < \arcsin \frac{1}{\sqrt{5}}$.

\emph{Case 2:} three of the great circles, say $G_1, G_2, G_3$, meet at the same pair of antipodes, and $G_4$ does not.

Let $\rho > 0$ denote the largest value such that each cell in the decomposition defined by the family $\{ G_1, G_2, G_3, G_4\}$ contains a spherical cap of radius $\rho$. Then each of the six lunes in the decomposition defined by $\{ G_1, G_2, G_3 \}$ contains two nonoverlapping spherical caps of radius $\rho$.
For any $0 < \varphi < \pi$, let us denote by $\rho(\varphi)$ the inradius of an isosceles spherical triangle of side lengths $\frac{\pi}{2}, \frac{\pi}{2}$ and $\varphi$. Our proof of Part (iii) of Theorem~\ref{main3} in Case 2 is based on the next lemma.

\begin{Lemma}\label{lem:thickness}
If a lune $L$ of angle $\varphi$ with $0 < \varphi < \pi$ contains two nonoverlapping spherical caps $C_1, C_2$ of radius $\rho > 0$, then $\rho \leq \rho(\varphi)$ with equality if and only if both $C_1$ and $C_2$ touch each side of $L$, and they touch each other at the midpoint of the bisector of $L$.
\end{Lemma}

\begin{proof}
In the proof, for $i=1,2$ we denote the center of $C_i$ by $\c_i$, the great circle containing $\widehat{\c_1\c_2}$ by $G$, and the bisector of the spherical segment $\widehat{\c_1\c_2}$ by $G'$ (cf. Figure~\ref{fig:thm3proof1}). Furthermore, following \cite{Lassak}, we call the spherical distance of the midpoints of the two sides of the lune $L \subset \S^2$ the \emph{thickness of $L$}, and the minimum thickness of a lune containing a given spherically convex body $K \subset \S^2$ the \emph{thickness of $K$}, denoted by $\Delta(K)$.  Note that, according to this definition, the thickness of a lune is equal to its angle.

Since $C_1 \cup C_2$ is contained in a lune $L$ and it does not contain any vertex of $L$, $C_1 \cup C_2$ is contained in an open hemisphere. Thus, its spherical convex hull $K$ exists, and a lune contains $C_1 \cup C_2$ if and only if it contains $K$. Let $L_0$ denote the lune obtained as the intersection of the two common supporting hemispheres of $C_1$ and $C_2$. First, we show that among the lunes containing $K$, the only lune with thickness $\Delta(K)$ is $L_0$. Indeed, by \cite[Claim 2]{Lassak}, if $L'$ is a lune with $K \subseteq L'$ and $\Delta(L')=\Delta(K)$, then the midpoints of both sides of $L'$ belong to $K$. Since $K$ is smooth, this yields that the spherical segment connecting the midpoints of the sides of $L'$ is a \emph{double normal} of $K$; i.e. it is a chord of $K$ perpendicular to the boundary of $K$ at both endpoints. Furthermore, as $K$ is the spherical convex hull of two distinct congruent spherical caps, the double normals of $K$ are exactly the intersections of $K$ with $G$ or $G'$. An elementary computation shows that the thickness of the latter lune is strictly smaller than the thickness of the first one, showing that $K \subset L'$, $\Delta(L')=\Delta(K)$ implies $L'=L_0$.

To finish the proof we observe that, by an elementary computation, the thickness of $L_0$ is a strictly increasing function of the spherical distance of the centers of $C_1$ and $C_2$, and thus, it is minimal if and only if $C_1$ and $C_2$ touch each other.
\end{proof}

Now, we prove Part (iii) in Case 2. For any $1 \leq i < j \leq 3$, let $\varphi_{ij}$ denote the angle of the two lunes in the cell decomposition of $\S^2$, bounded by half great circles in $G_i$ and $G_j$. Note that $\varphi_{12}+\varphi_{13}+\varphi_{23}=\pi$. By Lemma~\ref{lem:thickness}, we have $\rho \leq \min \{ \rho(\varphi_{12}), \rho(\varphi_{13}), \rho(\varphi_{23}) \}$. Thus, using the monotonicity properties of the function $\rho(\varphi)$, it follows that $\rho \leq \rho\left( \frac{\pi}{3} \right) = \arcsin \frac{1}{\sqrt{5}}$ with equality if and only if $\varphi_{ij} = \frac{\pi}{3}$ for all $1 \leq i < j \leq 3$, and $G_4$ is the polar of the common points of $G_1, G_2$ and $G_3$.

\begin{figure}[h]
\begin{minipage}[t]{0.48\linewidth}
\centering
\includegraphics[width=0.8\textwidth]{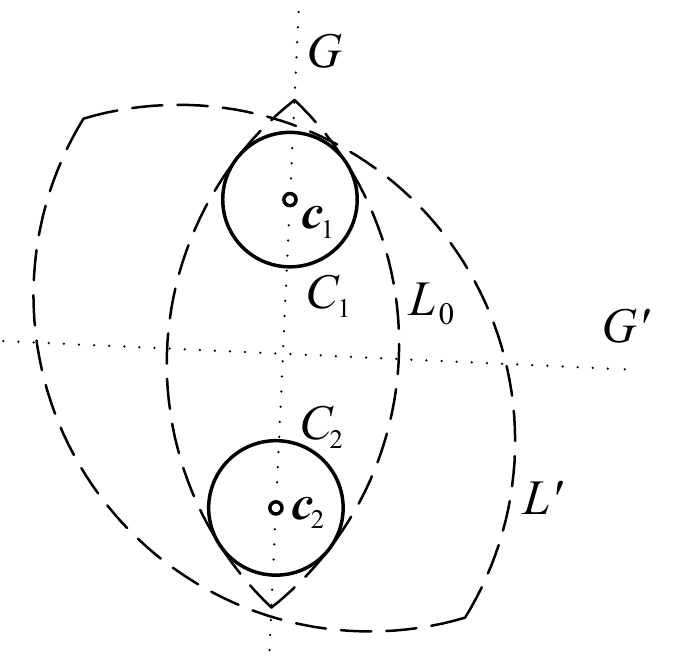}
\caption{An illustration for Lemma~\ref{lem:thickness}.}
\label{fig:thm3proof1}
\end{minipage}
\hglue1cm
\begin{minipage}[t]{0.48\linewidth}
\centering
\includegraphics[width=0.8\textwidth]{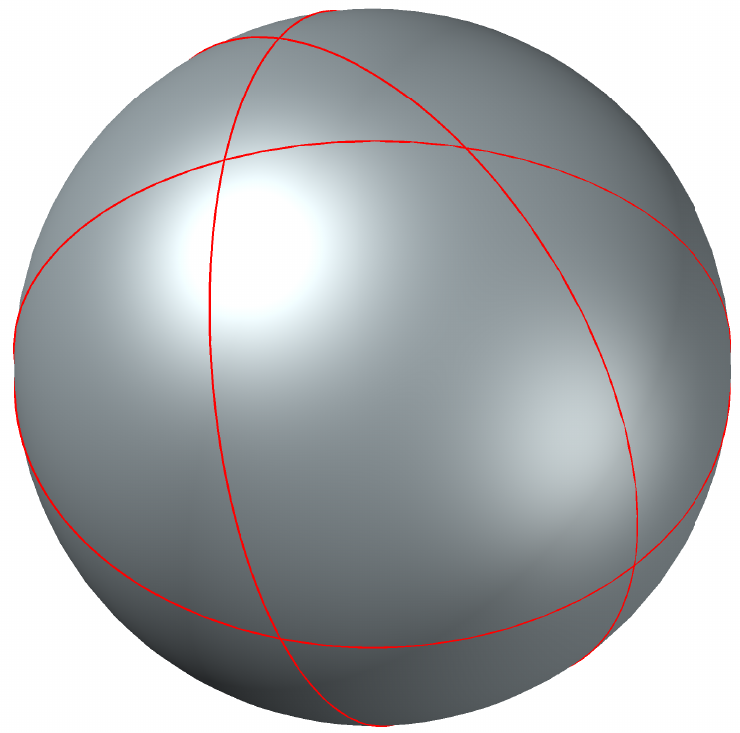}
\caption{An illustration for Case 3 in Section~\ref{subsec:part3}.}
\label{fig:thm3proof2}
\end{minipage}
\end{figure}

\emph{Case 3:} the family $\{ G_1, G_2, G_3, G_4 \}$ is nondegenerate, i.e., there is no point contained in at least three of the great circles (cf. Figure~\ref{fig:thm3proof2}).

In the following, let $r$ denote the minimal inradius of the cells of the arrangement.

An elementary computation shows that in the cell decomposition of $\S^2$ defined by this family, each great circle contains six sides, and there are exactly $14$ cells. On the other hand, by a result of Shannon (see \cite{Shannon}), this decomposition contains at least $8$ triangular cells. Let $Q_1, \ldots, Q_{14}$ denote the cells of the arrangement, and let $s_i$ denote the number of sides of $Q_i$. Without loss of generality, we assume that $s_1 = \ldots = s_8 = 3$. On the other hand, from the facts that every great circle contains $6$ sides of the arrangement and every side belongs to exactly two cells, we also have that $\frac{s_9+\ldots + s_{14}}{6} = 4$.

For all values of $i$, let $a_i$ denote the spherical area of the cell $Q_i$. By the spherical Dowker's theorem (\cite{FTF}, \cite{Mo}), we have that $a_i$ is greater than or equal to the area $a_r(s_i)$ of a regular $s_i$-gon circumscribed about a spherical cap of radius $r$. Furthermore, an elementary computation shows that $a_r(k)=2k\arccos\left( \cos r\sin \frac{\pi}{k}\right)-(k-2)\pi$, and that extending $a_r(k)$ to all  positive real values of $k$, we obtain a strictly convex function. By these observations and Jensen's inequality, we have
\begin{equation}\label{eq:fromDowker}
4 \pi = \sum_{i=1}^{14} a_i \geq 6 \left( 8 \arccos \left( \frac{1}{\sqrt{2}} \cos r\right)  - 2 \pi \right) + 8 \left( 6 \arccos \left( \frac{\sqrt{3}}{2} \cos r\right) - \pi \right).
\end{equation}

Solving (\ref{eq:fromDowker}) we obtain that $r \leq \arccos \frac{2}{\sqrt{5}} = \arcsin \frac{1}{\sqrt{5}}$. On the other hand, equality here occurs if and only if
$Q_1, \ldots, Q_8$ are regular triangles circumscribed about a spherical cap of radius $r$, and $Q_9, \ldots, Q_{14}$ are regular quadrangles circumscribed about a spherical cap of radius $r$. But then the sides of the triangle cells are strictly longer than the sides of the quadrangle cells, which contradicts the fact that in the cell decomposition defined by the four great circles, clearly there is a side which belongs to both a triangle and a quadrangle cell. This proves that $r < \arcsin \frac{1}{\sqrt{5}}$ for any arrangement of four great circles satisfying the condition in Case 3.

\subsection{Proof of Part (iv)}\label{subsec:part4}

Let $n>1$ pairwise distinct great circles dissect $\S^2$ into $N$ $2$-dimensional cells say, $Q_1, Q_2, \dots , Q_N$.
Assume that $\rho>0$ is chosen such that each $Q_i$ contains a spherical cap $S_i$ of angular radius $\rho$, where $1\leq i\leq N$. If $n\geq 5$, then Part (ii) implies that $\rho<\frac{\pi}{4}$ and therefore, Theorem~\ref{main2} applies to the TS-packing of the spherical caps $\F_N = \{ S_1, S_2, \ldots, S_N \}$ in $\S^2$. {As $2n\leq N$, therefore Theorem~\ref{main2} yields
\begin{equation}\label{5-4}
\frac{(2n)2\pi(1-\cos\rho)}{4\pi}\leq \delta(\F_N)\leq \delta(\rho) =\frac{1-\cos\rho}{1-\frac{\pi}{4\arcsin\left(\frac{1}{\sqrt{2}\cos\rho}\right)}}.  
\end{equation}
Finally, (\ref{5-4}) via a straightforward computation implies $\rho\leq \arccos\left(\frac{1}{\sqrt{2}\sin\left(\frac{n}{n-1}\frac{\pi}{4}\right)}\right)$, finishing the proof of Part (iv) of Theorem~\ref{main3}.

\section{Proof of Theorem~\ref{main4}}

\subsection{Proof of Part (i)}

Part (i) of Theorem~\ref{main4} is a special case of Part (i) of Theorem~\ref{thm:highdim} with a proof presented in Section~\ref{sec:highdim}.

\subsection{Proof of Part (ii)}

Apart from the equality case, our proof is a shortened version of the proof in \cite{Va}. Consider a family $\mathcal{F} = \{ G_1, G_2, G_3, G_4 \}$ of great circles. Depending on the position of the elements of $\mathcal{F}$, we distinguish three cases as in Subsection~\ref{subsec:part3}.

\emph{Case 1:} all of $G_1, \ldots, G_4$ meet at the same pair of antipodes.

Then, clearly, the circumradius of every cell is $\frac{\pi}{2}$.

\emph{Case 2:} three of the great circles, say $G_1, G_2, G_3$, meet at the same pair of antipodes, and $G_4$ does not.

Then $G_1, G_2$ and $G_3$ dissect $\S^2$ into six lunes, each of which is further dissected into two spherical triangles by $G_4$.
Consider any of the six lunes. Observe that if a half great circle arc is covered by the union of two spherical caps of radius $\rho$, then $\rho \geq \frac{\pi}{4}$. This implies that if a lune is covered by two spherical caps of radius $\rho$, then $\rho > \frac{\pi}{4}$. Thus, the maximum circumradius of the cells of the spherical mosaic generated by $\mathcal{F}$ is strictly greater than $\frac{\pi}{4}$.

\emph{Case 3:} the family $\mathcal{F}$ is nondegenerate, i.e., there is no point contained in at least three of the great circles (cf. Figure~\ref{fig:thm3proof2}).

Let $\rho_{\mathcal{F}}$ denote the maximum circumradius of all cells generated by $\mathcal{F}$. Since any two great circles in $\mathcal{F}$ intersect in two antipodal points, and any great great circle in $\mathcal{F}$ is decomposed into six circle arcs by these intersection points, the spherical mosaic $\mathcal{M}$ generated by $\mathcal{F}$ has $12$ vertices and $24$ edges. This implies, by Euler's theorem, that $\mathcal{M}$ has $14$ $2$-dimensional cells. Furthermore, since $\mathcal{F}$ contains four great circles, every cell in $\mathcal{M}$ is a triangle or a quadrangle. Every edge in $\mathcal{M}$ belongs to exactly two cells, yielding that $\mathcal{M}$ contains $8$ triangles and $6$ quadrangles.
For any $1 \leq i , j \leq 4$ and $i \neq j$, let $\p_{ij}$ and $-\p_{ij}$ denote the intersection points of $G_i$ and $G_j$.
Since every convex quadrangle has two diagonals, a simple counting shows that for any $\{ i,j,s,t \} = \{ 1,2,3,4\}$, the arcs $\widehat{\p_{ij}\p_{st}}, \widehat{\p_{st}(-\p_{ij})}, \widehat{(-\p_{ij})(-\p_{st})}$ and $\widehat{(-\p_{st} p_{ij})}$ are diagonals of quadrangle cells in $\mathcal{M}$, and every diagonal of a quadrangle cell can be obtained in this way.

Consider a diagonal of a quadrangle cell, say $\widehat{\p_{12}\p_{34}}$. Then the great circle $G$ containing this arc does not belong to $\mathcal{F}$, and hence, it contains exactly four vertices, namely $\p_{12}, \p_{34}, -\p_{12}, -\p_{34}$. As the arcs $\widehat{\p_{12}\p_{34}}, \widehat{\p_{34}(-\p_{12})}, \widehat{(-\p_{12})(-\p_{34})}$ and $\widehat{(-\p_{34})\p_{12}}$ are diagonals of quadrangle cells of $\mathcal{M}$, their lengths are at most $2\rho_{\mathcal{F}}$. Since their total length is $2\pi$, it readily follows that $\rho_{\mathcal{F}} \geq \frac{\pi}{4}$, implying the required inequality.

Now, we characterize the equality case. By the argument above, if $\rho_{\mathcal{F}} = \frac{\pi}{4}$, then every diagonal of every quadrangle in $\mathcal{M}$ is of length $\frac{\pi}{2}$, and the circumcenter of every quadrangle cell is the intersection point of its two diagonals, which coincides with the midpoints of the diagonals. Applying this observation for the diagonal $\widehat{\p_{12}\p_{34}}$, we obtain that the distance of the centers of the cells with diagonals $\widehat{\p_{12}\p_{34}}, \widehat{\p_{34}(-\p_{12})}$ is $\frac{\pi}{2}$, and $\p_{34}$ is the midpoint of the spherical segment connecting them. Since this observation can be applied for any diagonal of a quadrangle face, it follows that the centers of the six quadrangle faces are the vertices of a regular octahedron, and the vertices of $\mathcal{M}$ are the midpoints of the arcs connecting them. This yields the statement.


\subsection{Proof of Part (iii)}

Without loss of generality we may assume that $n>1$ {\it pairwise distinct} great circles dissect $\S^2$ into $N$ $2$-dimensional cells say, $Q_1, Q_2, \dots , Q_N$.  Let the number of sides (resp., spherical area) of $Q_i$ be denoted by $s_i$ (resp., $a_i$) for $1\leq i\leq N$.

\begin{Lemma}\label{max-number-of-cells}
$N\leq n(n-1)+2=n^2-n+2.$
\end{Lemma}
\begin{proof}
The following proof is by induction on $n$. Clearly, the statement holds for $n=2$. So, we assume that it holds for any tiling of $\S^2$ generated by at most $n-1$ ($\geq 2$) great circles. Next, let $G_1, \dots, G_{n-1}, G_{n}$ be pairwise distinct great circles in $\S^2$ that dissect $\S^2$ into $N$ $2$-dimensional cells. By induction, if $G_1, \dots, G_{n-1}$ dissect $\S^2$ into $N'$ $2$-dimensional cells, then 
\begin{equation}\label{5-1}
N'\leq (n-1)(n-2)+2=n^2-3n+4 .
\end{equation}
Next, observe that 
\begin{equation}\label{5-2}
N\leq N'+N'',
\end{equation}
where $N''$ denotes the number of $2$-dimensional cells of the tiling generated by  $G_1, \dots, G_{n-1}$ that are bisected by $G_n$. Clearly, $N''$ is equal to the number of circular arcs into which
$G_n$ is dissected by $G_1, \dots, G_{n-1}$ and therefore
\begin{equation}\label{5-3}
N''\leq2(n-1)=2n-2.
\end{equation}
Thus, (\ref{5-1}), (\ref{5-2}), and (\ref{5-3}) imply that $N\leq (n^2-3n+4)+(2n-2)=n^2-n+2$, which completes the proof of Lemma~\ref{max-number-of-cells}. 
\end{proof}

\begin{Lemma}\label{average side number}
$$2\leq\frac{\sum_{i=1}^Ns_i}{N}\leq 4.$$
\end{Lemma}
\begin{proof}
As the lower bound clearly holds, we are left with proving the upper bound. Using induction on $n$, we assume that the upper bound holds for any tiling of $\S^2$ induced by $n\geq 2$ pairwise distinct great circles. Next, let us take $n+1$ pairwise distinct great circles say, $G_1, G_2, \dots, G_n, G_{n+1}$ and assume that $G_1, G_2, \dots , G_n$ dissect $\S^2$ into the cells $Q_1, Q_2, \dots , Q_N$ with $Q_i$ having $s_i$ sides for $1\leq i\leq N$. The inductive assumption implies that
\begin{equation}\label{inductive-assumption}
\sum_{i=1}^Ns_i\leq 4N
\end{equation}
Without loss of generality we may assume that $G_{n+1}$ dissects each of the cells $Q_1, Q_2,\dots , Q_M$ into two cells and it is disjoint from the interiors of the remaining cells $Q_{M+1}, \dots , Q_N$. In particular, assume that the side numbers of the two cells into which $G_{n+1}$ dissects the cell $Q_j$ are labeled by $s_j'$ and $s_j''$ for $1\leq j\leq M$. Clearly, the total number of sides of the cells of the tiling generated by the great circles $G_1, G_2, \dots, G_n, G_{n+1}$ is equal to $ \sum_{j=1}^{M} (s_j'+s_j'')+\sum_{i=M+1}^Ns_i $ and our goal is to show that
\begin{equation}\label{induction-1}  
\sum_{j=1}^{M} (s_j'+s_j'')+\sum_{i=M+1}^Ns_i \leq 4(N+M).
\end{equation}
Finally, (\ref{induction-1}) follows from (\ref{inductive-assumption}) and the obvious observation that $\sum_{j=1}^{M} (s_j'+s_j'')\leq 4M+\sum_{j=1}^{M} s_j$ holds. This completes the proof of Lemma~\ref{average side number}.
\end{proof}

\begin{Lemma}\label{Dowker-covering}
Assume that each $Q_i$ is covered by a spherical cap of angular radius $0<R<\frac{\pi}{2}$, where $1\leq i\leq N$. Then
$$\frac{4\pi}{8\arctan\left(\frac{1}{\cos R}\right)-2\pi}\leq N.$$
\end{Lemma}
\begin{proof}
Let $A_R(s)$ be the spherical area of an $s$-sided regular spherical convex polygon which is inscribed a spherical cap of angular radius $0<R<\frac{\pi}{2}$. It is well known (\cite{FTF}) that
\begin{equation}\label{inscribed-regular}
a_i\leq A_R(s_i)
\end{equation}
holds for all $1\leq i\leq N$.
Next, let us recall the so-called spherical Dowker theorem (\cite{FTF}, \cite{Mo}) according to which the sequence $A_R(2), A_R(3), A_R(4), \dots$ is concave: $A_R(s-1)+A_R(s+1)\leq 2A_R(s)$, where $s=3,4,\dots$.
Thus, by extending $A_R(s)$ to all real values of $s\geq 2$ such that we obtain an increasing concave function, we get by (\ref{inscribed-regular}), Lemmas~\ref{average side number}, and Jensen's inequality that
\begin{equation}\label{lower estimate for total cell area}
4\pi=\sum_{i=1}^N a_i\leq \sum_{i=1}^NA_R(s_i)\leq NA_R\left(\frac{\sum_{i=1}^Ns_i}{N}\right)\leq N A_R(4)=N\left(8\arctan \left(\frac{1}{\cos R}\right)-2\pi \right),
\end{equation}
finishing the proof of Lemma~\ref{Dowker-covering}. \end{proof}
Hence, Lemmas~\ref{max-number-of-cells} and ~\ref{Dowker-covering} imply via a simple computation that $\arccos\left(\frac{1}{\tan\left(\left(1+\frac{2}{n^2-n+2}\right)\frac{\pi}{4}\right)  }\right)\leq R$ holds for all $n>2$. This completes the proof of Part (iii) of Theorem~\ref{main4}.

\subsection{Proof of Part (iv)}

Part (iv) of Theorem~\ref{main4} is a special case of Part (iii) of Theorem~\ref{thm:highdim} with a proof presented in Section~\ref{sec:highdim}.

\section{Proof of Theorem~\ref{thinnest-TS-covering}}

Let $\S^2$ be dissected by $n$ great circles into $N$ $2$-dimensional cells say, $Q_1, Q_2, \dots , Q_N$ such that each each $Q_i$ is covered by a spherical cap $C_i$ of $\mathcal{C}$ having angular radius $\rho$, where $C_i\neq C_j$ for any $1\leq i<j\leq N$. Then Lemma~\ref{Dowker-covering}
implies in a straightforward way that

$$\Delta({\rho})=\frac{\pi(1-\cos\rho)}{4\arctan\left(\frac{1}{\cos\rho}\right)-\pi}=\frac{2\pi(1-\cos\rho)}{8\arctan\left(\frac{1}{\cos\rho}\right)-2\pi}\leq\frac{N 2\pi(1-\cos \rho)}{4 \pi} = \frac{1-\cos \rho}{2} N\leq\delta(\mathcal{C}) ,$$
finishing the proof of Theorem~\ref{thinnest-TS-covering}.

\section{Proof of Theorem~\ref{thm:highdim}}\label{sec:highdim}

\subsection{Proof of Part (i)}

In the proof we need the fact that among spherical simplices contained in a given spherical ball, the regular ones inscribed in the ball have maximal volume. This fact is the main result of \cite{boroczky}. Nevertheless, by Lemma~\ref{lem:density}, we give a new and short proof of this statement.

\begin{Lemma}\label{lem:density}
Let $\mu$ be a rotationally symmetric, strictly decreasing, nonnegative function defined on a closed ball $B \subset \E^d$ centered at the origin, and let $S_{reg}$ be a regular simplex inscribed in $B$. Then, for any simplex $S \subset B$, we have $\int_S \mu(\p) \, d \p \leq \int_{S_{reg}} \mu(\p) \, d\p$, with equality if and only if $S$ is a regular simplex inscribed in $B$.
\end{Lemma}

\begin{proof}
We prove Lemma~\ref{lem:density} by induction on $d$. For $d=1$, the statement is trivial. Assume that it holds in $\E^{d-1}$.
Let $F$ be a facet of $S$ and let $H=H_0$ be the hyperplane through $F$. For any $t \in \Re$, let $H_t$ be the hyperplane parallel to $H$, at the signed distance $t$ from $H$ such that $H_t$ for $t > 0$ lies on the same side of $H$ as $S$. Let $L$ be the line through the origin $\o$ perpendicular to $H$, and let $\c_t$ denote the intersection point of $H_t$ and $L$. Finally, let $h>0$ be the height of $S$ with respect to the facet containing $H$. Then $S$ is a pyramid with base $F=H \cap S$.

\begin{figure}[ht]
\begin{center}
\includegraphics[width=0.65\textwidth]{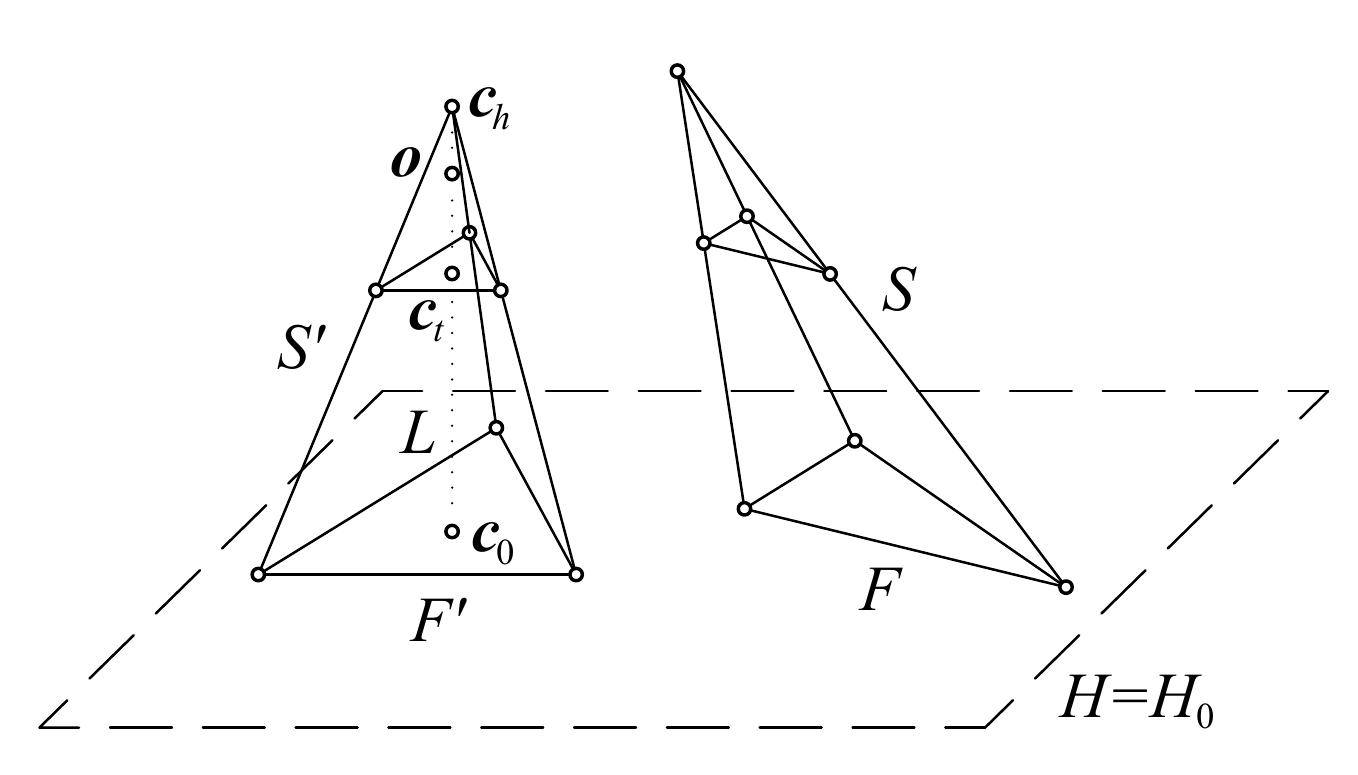}
\caption{An illustration for Lemma~\ref{lem:density}.}
\label{fig:thm5}
\end{center}
\end{figure}

Note that for any $t$, $\left. \mu \right|_{H_t}$ is a nonnegative, strictly decreasing function, which is rotationally symmetric about $\c_t$.
Let $F'$ be a regular $(d-1)$-dimensional simplex in $H_0$, centered at $\c_0$, such that the circumradii of $F'$ and $F$ are equal (cf. Figure~\ref{fig:thm5}). Let $S'=\conv (F' \cup \{ \c_h\})$. Clearly, for every value of $t$, the circumradius of $S' \cap H_t$ is equal to the circumradius of $S \cap H_t$.
Furthermore, by the induction hypothesis, for all values of $t$, we have
\begin{equation}
\int_{H_t \cap S} \mu(\p) \, d \p \leq \int_{H_t \cap S'} \mu(\p) \, d \p,
\end{equation}
from which Fubini's theorem yields $\int_{S} \mu(\p) \, d \p \leq \int_{S'} \mu(\p) \, d \p$.

On the other hand, equality implies that $F$ is a regular $(d-1)$-dimensional simplex centered at $\c_0$, and its apex is $\c_h$.
If $\int_{S} \mu(\p) \, d \p$ is maximal, all facets of $S$ satisfy this property, implying that $S$ is a regular simplex inscribed in $B$.
\end{proof}

\begin{Corollary}\label{Boroczky}
Among all spherical simplices contained in a spherical cap of angular radius $<\frac{\pi}{2}$ of $\S^{d-1}$ the regular inscribed simplices have the maximum volume.
\end{Corollary}
\begin{proof}
Consider the spherical cap $C$ in $\S^{d-1}$, centered at a point $\c$, and let $H$ denote the tangent hyperplane of $\S^{d-1}$ at $\c$. Let $\proj: C \to H$ denote the central projection with center $\o$ onto $H$. Then $\proj(C)$ is a Euclidean ball in $H$ centered at $\c$, and the volume element induced by the projection is a strictly decreasing function depending only on the distance from $\c$. Thus, the fact that among spherical simplices contained in $C$, the regular ones inscribed in $C$ have maximal volume follows from Lemma~\ref{lem:density}.
\end{proof}

Finally, we prove Part (i) of Theorem~\ref{thm:highdim}. Consider a family $\mathcal{F}_d$ of $d$ $(d-2)$-dimensional great spheres in $\S^{d-1}$. Observe that if there is a pair of antipodes contained in each great sphere of $\mathcal{F}_d$, then the circumradius of each cell is $\frac{\pi}{2}$. Hence, we may assume that the great spheres of $\mathcal{F}_d$ are in nondegenerate position, implying that they decompose $\S^{d-1}$ into $2^d$ spherical $(d-1)$-dimensional simplices. Notice that, if the great spheres of $\mathcal{F}_d$ are mutually orthogonal, then all $(d-1)$-dimensional cells are regular simplices of edge length $\frac{\pi}{2}$ with circumradius $\arccos\frac{1}{\sqrt{d}}$. Thus, the largest circumradius of the $2^d$ $(d-1)$-dimensional simplices into which the great spheres of $\mathcal{F}_d$ dissect $\S^{d-1}$ cannot be smaller than $\arccos\frac{1}{\sqrt{d}}$. Namely, if it is, then Corollary~\ref{Boroczky} would imply that the total volume of the $2^d$ simplices were less than the surface volume of $\S^{d-1}$, a contradiction. Finally, notice that using Corollary~\ref{Boroczky}
once again we get that the largest circumradius of the $2^d$ $(d-1)$-dimensional simplices into which the great spheres of $\mathcal{F}_d$ dissect $\S^{d-1}$ can be equal to $\arccos\frac{1}{\sqrt{d}}$ only if each of the  $2^d$ spherical $(d-1)$-dimensional simplices is a regular one of edge length $\frac{\pi}{2}$, finishing the proof of Part (i) of Theorem~\ref{thm:highdim}.

\subsection{Proof of Part (ii)}

\begin{Lemma}\label{upper-estimate-for-cell-numbers}
Let $n\geq d$ pairwise distinct $(d-2)$-dimensional great spheres of $\S^{d-1}$ ($d\geq 2$)  dissect $\S^{d-1}$ into $N$ $(d-1)$-dimensional cells. Then 
\begin{equation}\label{upper-bound-for-N}
N\leq 2\left(\binom{n-1}{0}+\binom{n-1}{1}+\dots +\binom{n-1}{d-1}\right)
\end{equation}
\end{Lemma}
\begin{proof}
Let $G_1, G_2, \dots , G_n$ be pairwise distinct $(d-2)$-dimensional great spheres of $\S^{d-1}$ ($d\geq 2$) that dissect $\S^{d-1}$ into $N=2k$ $(d-1)$-dimensional cells. Moreover, let $H_1$ be the hyperplane spanned by $G_1$ in  $\E^d$ with normal vector ${\bf n}_1\in \S^{d-1}$ and let $H$ be the hyperplane tangent to $\S^{d-1}$ at the point ${\bf n}_1$. Clearly, $H_1$ and $H$ are parallel and the images under the central projection from $\o$ onto $H$ of the $(d-2)$-dimensional great spheres $G_2, \dots , G_n$ are hyperplanes say, $G'_2, \dots , G'_n$ of the $(d-1)$-dimensional Euclidean space $H$. It follows that the number of $(d-1)$-dimensional cells into which $G'_2, \dots , G'_n$ dissect $H$ is equal to $k$. Now, recall the well-known statement (\cite{Bu}) according to which $k\leq \left(\binom{n-1}{0}+\binom{n-1}{1}+\dots +\binom{n-1}{d-1}\right)$. This finishes the proof of Lemma~\ref{upper-bound-for-N}.
\end{proof}
Assume that there exists an arrangement of $n>d$ (pairwise distinct) $(d-2)$-dimensional great spheres in $\S^{d-1}$ ($d>3$) which dissect $\S^{d-1}$ into the $(d-1)$-dimensional cells $Q_1, Q_2, \dots , Q_N$  such that each $Q_i$ is covered by a spherical cap of radius $R:=R_{\rm gs}(n, \S^{d-1})<\frac{\pi}{2}$. Then, via Lemma~\ref{upper-estimate-for-cell-numbers} the inequality (\ref{upper-bound-for-N}) holds. Next, recall the theorem of Glazyrin (Theorem 1.2 in  \cite{Gl}) stating that for any covering of $\S^{d-1}$ by closed spherical caps of angular radius less than $\frac{\pi}{2}$, the sum of the Euclidean radii of the caps is greater than $d$. Thus, we have
\begin{equation}\label{Glazyrin}
d<N\sin R .
\end{equation}
Hence, (\ref{upper-bound-for-N}) and (\ref{Glazyrin}) imply in a straightforward way that $  \arcsin\left(\frac{d}{2\sum_{i=0}^{d-1}\binom{n-1}{i}}\right)<R$. This completes the proof of Part (ii) of Theorem~\ref{thm:highdim}.

\subsection{Proof of Part (iii)}

By \cite{ball}, for any $0 < \delta < \frac{\pi}{2}$, there is a $\delta$-net $X$ of cardinality at most $\left( \frac{4}{\sin \delta} \right)^d$ on $\S^{d-1}$, or in other words, there is a set $X \subset \S^{d-1}$ of cardinality at most $\left( \frac{4}{\sin \delta} \right)^d$ such that every point of $\S^{d-1}$ is at spherical distance at most $\delta$ from a point of $X$. Thus, for any $n>d\geq 3$, there is a set $X_n \subset \S^{d-1}$ of cardinality at most $n$ such that $X_n$ is a $\delta$-net with $\delta = \arcsin \left( \frac{4}{\sqrt[d]{n}} \right)$. Then every point $\x \in X_n$ is the center of a closed hemisphere bounded by the $(d-2)$-dimensional great sphere $H_{\x}$. Let $\F_n = \{ H_{\x}: \x \in X_n \}$, and let $C$ denote an arbitrary $(d-1)$-dimensional cell of the decomposition of $\S^{d-1}$ by the elements of $\F_n$. 

Let $Y_n = X_n \cup (-X_n)$, and consider an arbitrary point $\p$ in the interior of $C$. We are going to show that the closed spherical cap of radius $2 \delta$ and center $\p$ contains $C$, and note that this implies that the diameter of $C$ is at most $2 \delta$. Let $H_{\p}$ denote the $(d-2)$-dimensional great sphere bounding the closed hemisphere centered at $\p$, and let the set $Y_n(C)$ consist of the points of $Y_n$ separated by $H_{\p}$ from $C$. Note that $Y_n(C)$ is contained in an open hemisphere of $\S^{d-1}$, and its spherical convex hull is the polar of $C$. We denote this set by $C^*$. Clearly, the following are equivalent, where by a $\tau$-neighborhood of a compact set $Z$ we mean the set of points at distance at most $\tau$ from a point of $Z$:
\begin{itemize}
\item the spherical ball of radius $2 \delta$ and center $\p$ contains $C$;
\item the spherical ball of radius $\frac{\pi}{2} - 2\delta$ and center $\p$ is contained in $C^*$;
\item every boundary point of $C^* \cup (-C^*)$ is at spherical distance at most $2 \delta$ from $H_{\p}$, i.e., the $(2\delta)$-neighborhood of $H_{\p}$ contains the boundary of $C^*$;
\item the $(2 \delta)$-neighborhood of $C^*$ contains $H_{\p}$.
\end{itemize}
Now, let $\q$ be an arbitrary point of $H_{\p}$, and let $\q'$ denote the point on $\widehat{(-\p)\q}$ at spherical distance $\delta$ from $H_{\p}$. Since $Y_n$ is a $\delta$-net, the $\delta$-neighborhoods of $C^*$ and $-C^*$ cover $\S^{d-1}$. On the other hand, as $-C^*$ is strictly separated from $\q'$ by $H_{\p}$, it follows that $\q'$ lies in the $\delta$-neighborhood of $C^*$. Thus, by the triangle inequality, $\q$ lies in the $(2\delta)$-neighborhood of $C^*$.

We have shown that the diameter of $C$ is at most $2\delta$. Now, by Jung's theorem for spherical space \cite{jung}, any compact set $Z \subset \S^{d-1}$ of diameter $D$ can be covered by a spherical ball of radius $\arcsin \left( \sqrt{ \frac{2d-2}{d} } \sin \frac{D}{2} \right)$. Replacing $D$ in this formula with $2 \delta = 2 \arcsin \left( \frac{4}{\sqrt[d]{n}} \right)$ completes the proof of Theorem~\ref{thm:highdim}.

\bigskip


\noindent K\'aroly Bezdek \\
\small{Department of Mathematics and Statistics, University of Calgary, Canada}\\
\small{Department of Mathematics, University of Pannonia, Veszpr\'em, Hungary\\
\small{E-mail: \texttt{bezdek@math.ucalgary.ca}}

\bigskip

\noindent and

\bigskip

\noindent Zsolt L\'angi \\
\small{MTA-BME Morphodynamics Research Group and Department of Geometry}\\ 
\small{Budapest University of Technology and Economics, Budapest, Hungary}\\
\small{\texttt{zlangi@math.bme.hu}}

\end{document}